\newif\ifspringer
\newif\ifelsevier
\elseviertrue 

\ifspringer
\RequirePackage{fix-cm} 
\documentclass[smallextended,nospthms]{svjour3}
\usepackage[paperwidth=15.51cm,paperheight=23.51cm,left=1.75cm,right=1.75cm,top=2cm,bottom=2cm]{geometry} 
\fi

\ifelsevier
\documentclass[preprint,3p,times,11pt]{elsarticle}  
\fi

\usepackage[english]{babel}
\usepackage[utf8]{inputenc}
\usepackage{pdflscape} 
\usepackage{amssymb}
\usepackage{amsmath}
\usepackage{listings}
\usepackage{fancyvrb}
\usepackage{amsthm}
\usepackage{mathtools}
\ifspringer
\usepackage{bm} 
\fi
\usepackage[normalem]{ulem}
\usepackage{enumitem}
\usepackage{marvosym}
\usepackage[dvipsnames,svgnames]{xcolor}
\usepackage{array}
\usepackage{multirow}
\usepackage{bbm} 
\ifspringer
\usepackage{mathptmx} 
\fi

\usepackage{graphicx}
\ifspringer
\usepackage[FIGBOTCAP,TABTOPCAP,scriptsize]{subfigure} 
\fi
\ifelsevier
\usepackage[FIGBOTCAP,TABTOPCAP]{subfigure} 
\fi
\usepackage{tikz}
\usetikzlibrary{calc,shapes,arrows,backgrounds}
\tikzstyle{every picture}+=[remember picture]
\usepackage{pgfplots} 

\usepackage{hyperref}
\hypersetup{colorlinks=true,linkcolor=blue,citecolor=blue,filecolor=blue,urlcolor=blue}
\usepackage[all]{hypcap}

\mathtoolsset{showonlyrefs} 
\allowdisplaybreaks 
\ifspringer
\smartqed 
\fi
\ifelsevier
\biboptions{sort&compress,numbers}
\fi

\usepackage{xspace}
\makeatletter
\DeclareRobustCommand\onedot{\futurelet\@let@token\@onedot}
\newcommand{\@onedot}{\ifx\@let@token.\else.\null\fi\xspace}
 
\newcommand{\ie}{{i.e}\onedot}

\makeatother

\newtheorem{prop}{Proposition}
\newtheorem{rem}{Remark}

\newtheorem{lem}{Lemma}
\theoremstyle{definition}
\newtheorem{defn}{Definition}

\newtheorem{exmp}{Example}

\newcommand{\RR}{\mathbb{R}}

\newcommand{\vectbi}[1]{\boldsymbol{#1}} 
\newcommand{\matri}[1]{#1}               
\newcommand{\vect}[1]{\vectbi{#1}}
\newcommand{\matr}[1]{\matri{#1}}

\DeclareMathOperator{\Span}{span}
\renewcommand{\leq}{\leqslant}
\renewcommand{\geq}{\geqslant}

\ifspringer
\DeclareMathAlphabet{\mathcalb}{OMS}{cmsy}{b}{n} 
\DeclareMathAlphabet{\mathcal}{OMS}{cmsy}{m}{n} 
\newcommand{\bUU}{\mathcalb{U}}
\fi
\ifelsevier
\newcommand{\bUU}{\vect{\mathcal{U}}}
\fi

\newcommand{\UU}{\mathcal{U}}

\newcommand{\bM}{\vect{M}}

\newcommand{\bDelta}{\vect{\Delta}}
\newcommand{\bR}{\vect{R}}

\newcommand\restr[2]{\ensuremath{\left.#1\right|_{#2}}}

\newlength{\casesvsep}
\newcommand{\figpath}{figure}
\graphicspath{{\figpath/}}

\makeatletter
\newcommand*{\shifttext}[2]{%
\settowidth{\@tempdima}{#2}%
\makebox[\@tempdima]{\hspace*{#1}#2}%
}
\makeatother

\newcommand{\ann}[1]{{\color{orange}#1}}

\makeatletter
\newcommand\footnoteref[1]{\protected@xdef\@thefnmark{\ref{#1}}\@footnotemark}
\makeatother

\ifspringer

\fi
\ifelsevier

\fi

\ifspringer
\journalname{\dots}
\fi

\ifelsevier
\journal{Applied Mathematics and Computation}
\fi

\makeatletter
\providecommand{\doi}[1]{%
  \begingroup
    \let\bibinfo\@secondoftwo
    \urlstyle{rm}%
    \href{http://dx.doi.org/#1}{%
      doi:\discretionary{}{}{}%
      \nolinkurl{#1}%
    }%
  \endgroup
}
\makeatother

\makeatletter
\def\ps@pprintTitle{%
  \let\@oddhead\@empty
  \let\@evenhead\@empty
  \def\@oddfoot{\reset@font\hfil\thepage\hfil}
  \let\@evenfoot\@oddfoot
}
\makeatother

\begin{document}



\newcommand{\titletext}{A practical criterion for the existence of optimal piecewise Chebyshevian spline bases}

\newcommand{\titlerunningtext}{\ann{\titletext}}

\newcommand{\abstracttext}{
A piecewise Chebyshevian spline space is a space of spline functions having pieces in different Extended Chebyshev spaces and where the continuity conditions between adjacent spline segments are expressed by means of connection matrices.
Any such space is suitable for design purposes when it possesses an optimal basis (\ie a totally positive basis of minimally supported splines) and when this feature is preserved under knot insertion.
Therefore, when any piecewise Chebyshevian spline space where all knots have zero multiplicity enjoys the aforementioned properties, then so does any spline space with knots of arbitrary multiplicity obtained from it.

In this paper, we provide a practical criterion and an effective numerical procedure to determine whether or not a given piecewise Chebyshevian spline space with knots of zero multiplicity is suitable for design. Moreover, whenever it exists, we also show how to construct the optimal basis of the space.
}

\ifelsevier
\newcommand{\separ}{\sep}
\fi

\newcommand{\keytext}{
Generalized spline spaces \separ Extended Chebyshev (piecewise) spaces \separ Optimal normalized totally positive basis \separ Transition functions \separ Weight functions \separ Geometric design 
}

\newcommand{\MSCtext}{
65D07 \separ 65D17 \separ 41A15 \separ 68W25
}


\ifelsevier
\begin{frontmatter}

\title{\titletext}

\author[label1]{Carolina Vittoria Beccari}
\ead{carolina.beccari2@unibo.it}
\author[label1]{Giulio Casciola}
\ead{giulio.casciola@unibo.it}


\address[label1]{Department of Mathematics, University of Bologna,
Piazza di Porta San Donato 5, 40126 Bologna, Italy}

\begin{abstract}
\abstracttext
\end{abstract}


\end{frontmatter}
\fi

\section{Introduction}
\label{sec:intro}
Extended Chebyshev spaces (EC-spaces) represent a natural generalization of polynomial spaces. They contain transcendental functions and provide additional degrees of freedom, that can be exploited to control the behavior of parametric curves and to accomplish shape-preserving approximations.
While Chebyshevian splines are piecewise functions whose pieces belong to the same EC-space \cite{Schu2007}, piecewise functions having sections in different EC-spaces are called piecewise Chebyshevian splines. The latter are of great interest in Geometric Design and Approximation for their capacity to combine the local nature of splines with the diversity of shape effects provided by the wide range of known EC-spaces.

Unfortunately, in general, there is no guarantee that a piecewise Chebyshevian spline space can have a real interest for applications.
In particular, the existence of a normalized, totally positive B-spline basis is essential in order to have computational stability and good approximation properties. Such a basis is the Optimal Normalized Totally Positive basis (ONTP basis for short), in the sense of the B-basis \cite{CP96}.
Moreover, not only the ONTP basis shall exist in the spline space itself, but also in all other spaces derived from it by insertion of knots.
Besides being crucial for the development of most geometric modeling algorithms, the latter feature allows for the existence of a multiresolution analysis and permits local refinement for solving PDEs.
Therefore, a spline space is suitable for design when it possesses the ONTP basis and when this property is preserved under knot insertion.

Piecewise Chebyshevian spline spaces were introduced in the seminal paper by Barry
\cite{Barry1996}. In that framework, the continuity conditions between adjacent spline pieces are expressed in terms of connection matrices linking the appropriate number of left and right generalized derivatives. By exploiting different theoretical machineries, first Barry and later Mühlbach and coauthors \cite{BMuhl2003,Muhl06,MT2006} proved that, if all the connection matrices are totally positive, then the corresponding spline space has a B-spline basis.
Later, Mazure showed that total positivity of the connection matrices is a far too restrictive assumption \cite{Maz2001}.
Moreover she demonstrated the equivalence between the existence of blossoms \cite{Maz2004b,Maz2009} and the existence of a B-spline basis both in the spline space itself and in all other spaces derived from it by knot insertion.
When blossoms exist, Mazure calls a space ``good for design", the terminology being also motivated by the fact that blossoms enable to easily develop all classical geometric design algorithms such as evaluation, knot-insertion and subdivision, and guarantee the existence of the ONTP basis.

From a practical point of view, however, it is not easy to check whether or not blossoms do exist.
To overcome this difficulty, Mazure introduced the notion of Extended Chebyshev Piecewise spaces (ECP-spaces) and generalized the classical theory of EC-spaces to the piecewise setting \cite{Maz2011b}.

Given a closed and bounded real interval $I$, a (single) $m$-dimensional EC-space $\UU_m \subset C^{m-1}(I)$ containing constants is good for design if and only if $D\UU_m$ is an EC-space \cite{CMP2003}.
For null spaces of linear differential operators with constant coefficients the latter property is verified whenever the length of the interval $I$ does not exceed a maximal length, which is known or can be computed for a wide class of EC-spaces of interest for applications \cite{CMP2003,MP2010,BM2012}.

In the piecewise setting, we shall take an increasing partition $a=x_0<x_1<\dots<x_k<x_{k+1}=b$ of an interval $[a,b]$, and assume that, for $i=0,\dots,k$, $\UU_{i,m}$ is an $m$-dimensional EC-space on $[x_i,x_{i+1}]$ containing constants.
Let us consider the $m$-dimensional space of all piecewise $C^{m-1}$ functions $F$ defined separately on each interval $[x_i,x_{i+1}]$ such that the restriction of $F$ to $[x_i,x_{i+1}]$, denoted by $F^{[i]}$, belongs to $\UU_{i,m}$ for all $i=0,\dots,k$ and such that, for $i=1,\dots,k$, the left and right derivatives of $F$ are connected by means of the relation
\begin{equation}\label{eq:PCn}
\left(D^0 F^{[i]}(x_i),D^1 F^{[i]}(x_i),\dots,D^{(m-1)} F^{[i]}(x_i)\right)^T=\\
R_i \left(D^0 F^{[i-1]}(x_i),D^1 F^{[i-1]}(x_i),\dots,D^{(m-1)} F^{[i-1]}(x_i)\right)^T,
\end{equation}
where each $R_i$, $i=1,\dots,k$, is a proper \emph{connection matrix} of order $m\times m$.
We call any such space a spline space with knots of zero multiplicity.
Mazure \cite{Maz2011b} proved that, in analogy with the non-piecewise case, if the piecewise space $DS$ is an ECP-space, then the spline space $S$ with knots of zero multiplicity is good for design and so is any space obtained from it by insertion of knots.
In the same paper, she also showed that systems of piecewise weight functions generate ECP-spaces just like systems of weight functions produce EC-spaces. More precisely, she demonstrated that a spline space with knots of zero multiplicity is an ECP-space if and only if it can be generated by a system of piecewise weight functions and
she showed how to build all the infinitely many possible such systems.


Given an arbitrary partition of an interval $[a,b]$, a sequence of EC-spaces and a sequence of connection matrices, in this paper we provide a practical method and an effective numerical procedure to determine whether or not the corresponding spline space with knots of zero multiplicity is suitable for design.
To this aim, after introducing the fundamental notions and results (Section \ref{sec:context}), in Section \ref{sec:suff_cond} we illustrate how to compute a set of functions, called \emph{transition functions}, that, when suitably combined, give rise to the ONTP basis of the space (provided it exists).
Exploiting the transition functions we construct a particular sequence of piecewise functions and show that, if the space in question is suitable for design, then they are an associated system of piecewise weight functions.
In Section \ref{sec:test} we develop a computationally efficient algorithm during which the sequence of candidate weight functions is recursively computed and, at each step, a proper test on the transition functions is performed to decide whether to continue or not. In this way, if the procedure reaches the final step, then the given space is suitable for design.

Because when a spline space with knots of zero multiplicity is suitable for design then so is any space obtained from it by knot insertion, the proposed procedure also yields sufficient conditions for the existence of ONTP bases in piecewise Chebyshevian spline spaces with knots of arbitrary multiplicity.
Moreover, it provides a tool to exploit spline spaces with knots of zero multiplicity themselves for design purposes.
Indeed, it has already been demonstrated that the additional degrees of freedom provided
by the connection matrices can be exploited as shape parameters \cite{Maz-Laurent98}.
The effective numerical procedure proposed in the present paper can be used to interactively tune these parameters so as to control the shape of parametric curves, while staying in the class of suitable spaces.

Finally, we would like to mention the recent application of spline spaces with knots of zero multiplicity to the construction of locally supported spline interpolants \cite{BCR2013a,ABC2013a}.


\section{Basic notions and notation}
\label{sec:context}
We start by introducing Extended Chebshev spaces (EC-spaces), that represent the building blocks of piecewise Chebyshevian splines.

\begin{defn}[Extended Chebyshev space]\label{def:ECspace}
Let $I\subset\RR$ be a closed bounded interval. An $m$-dimensional space $\UU_m\subset C^{m-1}(I)$, $m\geq1$, is an \emph{Extended Chebyshev space} (EC-space, for short) on $I$ if any nonzero element of $\UU_m$ vanishes at most $m-1$ times in $I$, counting multiplicities as far as possible for $C^{m-1}$ functions (that is, up to $m$), or, equivalently, if any Hermite interpolation problem in $m$ data in $I$ has a unique solution in $\UU_m$.
\end{defn}

Throughout the paper we adopt the following notation. Let $I=[a,b]\subset \RR$ be a closed and bounded interval and $\bDelta\coloneqq\left\{x_i, i=1,\dots,k\right\}$ a sequence of points, such that $a \equiv x_{0} < x_{1} < \ldots <x_{k} < x_{k+1} \equiv b$. In particular $\bDelta$ determines a sequence of subintervals of the form $I_i\coloneqq\left[x_i,x_{i+1}\right]$, $i=0,\dots,k$.

We shall say that $F$ is a \emph{piecewise function} on $(I,\bDelta)$ if $F$ is defined separately on each interval $I_i$ of $\bDelta$, $i=0,\dots,k$, meaning that $F(x_i^-)$ and $F(x_i^+)$ are defined, but they may be different. In analogy, we shall say that $F$ is a \emph{piecewise $C^n$} function on $(I,\bDelta)$ if $F$ is $C^n$ in each interval $I_i$.

Let us now consider an ordered set of $m$-dimensional spaces $\bUU_m \coloneqq\{\UU_{0,m},\dots,\UU_{k,m}\}$, such that every $\UU_{i,m}$ is an EC-space on the interval $I_i$, for $i=0,\dots,k$.
Moreover, let us associate to the elements of $\bDelta$ a multiplicity vector, namely
a vector of positive integers $\bM\coloneqq\left(m_1,\dots,m_k\right)$, such that $0 \leq m_i \leq m-1$ for all $i=1,\dots,k$, and
a sequence of connection matrices $\bR \coloneqq \{R_i, i=1,\dots,k\}$, where $R_i$ is lower triangular, of order $m-m_i$, has positive diagonal entries and has first row and column equal to $(1,0,\dots,0)$.
Hence the space of piecewise Chebyshevian spline functions \emph{based on} $\bUU_{m}$, which we indicate by $S(\bUU_m,\bM,\bDelta,\bR)$, is defined as follows.

\begin{defn}[Piecewise Chebyshevian splines]\label{def:PCS}
\qquad We define the set of piecewise Chebyshevian splines\\ $S(\bUU_m,\bM,\bDelta,\bR)$ based on $\bUU_m$ with knots $\bDelta \coloneqq \{x_1,\dots,x_k\}$ of multiplicities $\bM \coloneqq (m_1,\dots,m_k)$ and connection matrices $\bR \coloneqq \{R_1,\dots,R_k\}$ as
the set of all piecewise $C^{m-1}$ functions $s$ on $(I,\bDelta)$ such that:
\begin{itemize}
\item[i)] the restriction of $s$ to $I_i$, denoted by $s^{[i]}$, belongs to $\UU_{i,m}$, for $i=0,\dots,k$;
\item[ii)] $\left(D^0 s^{[i]}(x_i),D^1 s^{[i]}(x_i),\dots,D^{(m-m_i-1)} s^{[i]}(x_i)\right)^T =
R_i \left(D^0 s^{[i-1]}(x_i),D^1 s^{[i-1]}(x_i),\dots,D^{(m-m_i-1)} s^{[i-1]}(x_i)\right)^T$,
 $\ i=1,\dots,k.$
\end{itemize}
\end{defn}

The requirement that the first row and column of each matrix $R_i$ be equal to the vector $(1,0,\dots,0)$ guarantees the continuity of the splines defined above.
When all the matrices $R_i$ are the identity matrix, such splines are \emph{parametrically continuous}. Conversely, if the matrices $R_i$ are not the identity matrix, the corresponding spline space is a space of \emph{geometrically continuous} piecewise Chebyshevian splines.

In the particular case where the multiplicity vector has all elements equal to zero, we denote a piecewise Chebyshevian spline space by $S(\bUU_m,\bDelta,\bR)$ and call it a \emph{spline space with knots of zero multiplicity}.
Therefore, a spline space with knots of zero multiplicity is an $m$-dimensional space obtained by joining a number of
different $m$-dimensional EC-spaces with proper connection matrices. We also use the notation $DS(\bUU_m,\bDelta,\bR)\coloneqq \{DF \, | \, F \in S(\bUU_m,\bDelta,\bR)\}$.

The fact that all matrices in Definition \ref{def:PCS} are assumed to be lower triangular and have positive diagonal elements is essential to count zeros as well as to make Rolle's theorem valid in the piecewise context \cite{Maz2005}.
In fact, the regularity and lower triangular structure of the connection matrices entail that, for $i = 1,\dots,k$,
$x_i^+$ is a zero of multiplicity $r\leq m$ of
a given function $F\in S(\bUU_m,\bDelta,\bR)$ if and only if so is $x_i^-$.
Under the stated assumptions on the connection matrices, we can
therefore introduce the number of zeros $Z_{m}(F)$ as the total number of zeros of $F$ in $I$, counting multiplicities up to $m$, as for functions in $C^{m-1}(I)$.
This allows us to generalize the notion of EC-space to the piecewise setting as follows \cite{Maz2005}.

\begin{defn}[ECP-space]\label{def:ECP}
A spline space $S(\bUU_m,\bDelta,\bR)$ is an \emph{Extended Chebyshev Piecewise space} on $(I,\bDelta)$ (ECP-space for short) if any
of the two following properties is satisfied:
\begin{enumerate}
\item any nonzero element $F\in S(\bUU_m,\bDelta,\bR)$ satisfies $Z_{m}(F)\leq m-1$; 
\item any Hermite interpolation problem has a unique solution in $S(\bUU_m,\bDelta,\bR)$ in the sense that, for any positive integers $\mu_1,\dots,\mu_h$, such that $\sum\nolimits_{j=1}\nolimits^h\mu_j = m$, any pairwise distinct $\tau_1,\dots,\tau_h$ $\in I$, any $\epsilon_1,\dots,\epsilon_h \in \{+,-\}$, and any real numbers $\alpha_{j,r}$, $r=0,\dots,\mu_j-1$, $j = 1,\dots,h$, there exists a unique element $F$ of $S(\bUU_m,\bDelta,\bR)$ such that
    \begin{equation}\label{eq:Hermite}
    F^{(r)}\left(\tau_j^{\epsilon_j}\right)=\alpha_{j,r}, \qquad 0 \leq r\leq \mu_j-1, \quad 1\leq j\leq h.
    \end{equation}
\end{enumerate}
\end{defn}

From the above definition, it can be seen that ECP-spaces share with polynomial and Extended Chebyshev spaces the same bound of zeros for their non-zero elements. Moreover, the class of ECP-spaces is closed under integration and multiplication by positive piecewise functions \cite{Maz2011b}.
Exploiting this property, in the same work, it was proved that systems of piecewise weight functions produce ECP-spaces just like
systems of weight functions classically produce EC-spaces.
More precisely, a \emph{system of piecewise weight functions} is a sequence of piecewise functions $\{w_0,\dots,w_{m-1}\}$ on $(I,\bDelta)$ such that, for all $j=0,\dots,m-1$, $w_j$ is positive and piecewise $C^{m-j-1}$.
%
Based on it, we can define a sequence of piecewise linear differential operators, or \emph{generalized derivatives}, as
\begin{equation}\label{eq:der_gen}
L_0 F \coloneqq \frac{F}{w_0}, \qquad
L_j F \coloneqq \frac{1}{w_j}D L_{j-1} F, \quad j=1,\dots,m-1,
\end{equation}
where $D$ denotes ordinary differentiation (meant piecewisely) and $F$ is any piecewise $C^{m-1}$ function on $(I,\bDelta)$. The following result holds \cite{Maz2011b}.

\begin{prop}\label{prop:ECP_def}
Let $\{w_0,\dots,w_{m-1}\}$ be a system of piecewise weight functions associated with piecewise differential operators $L_0,L_1,\dots,L_{m-1}$. Then the set of all piecewise $C^{m-1}$ functions $F$ on $(I,\bDelta)$ such that
\begin{enumerate}
\item[i)] $L_{m-1}(F)$ is constant on $I$;
\item[ii)] $L_j F^{[i]}(x_i)= L_j F^{[i-1]} (x_i)$,\; $i = 1,\dots,k$, \;$j = 0,\dots,m-1$;
\end{enumerate}
is an $m$-dimensional ECP-space on $(I,\bDelta)$.
\end{prop}


%


We shall use the above result in the following way. Given a spline space $S(\bUU_m,\bDelta,\bR)$ with knots of zero multiplicity, if there does exist a system of piecewise weight functions $\{w_0,\dots,w_{m-1}\}$ such that any spline in the space satisfies \emph{i)} and \emph{ii)}, then
$S(\bUU_m,\bDelta,\bR)$ is an ECP-space. In this case we say that $S(\bUU_m,\bDelta,\bR)$ is associated with the system of piecewise weight functions and write $S(\bUU_m,\bDelta,\bR)=ECP(w_{0},\dots,w_{m-1})$. 

The property of being an ECP-space is closely related to the concept of Bernstein basis.


\begin{defn}[Bernstein-like and Bernstein basis]
\label{def:Bernst_basis}
A \emph{Bernstein-like basis on $I=[a,b]$} is a sequence of functions $\{B_{\ell,m}, \ \ell=1,\dots,m\}$
in $S(\bUU_m,\bDelta,\bR)$ such that $B_{\ell,m}$ vanishes exactly $\ell-1$ times at $a$ and exactly $m-\ell$ times at $b$ and is positive on $(a,b)$. A Bernstein-like basis is said to be a \emph{Bernstein basis on $[a,b]$} if it is normalized, meaning that $\sum^{m}_{\ell=1} B_{\ell,m}(x) = 1$, $\forall x\in[a,b]$.
\end{defn}

On any interval $[c,d]\subset[a,b]$ a Bernstein or Bernstein-like basis is a basis in the restriction of  $S(\bUU_m,\bDelta,\bR)$ to $[c,d]$ that satisfies the requirements of Definition \ref{def:Bernst_basis} at $c$ and $d$.

The theory of EC- and ECP-spaces and the study of their link with the existence of Bernstein-type bases were developed by Mazure and we refer the reader to \cite{Maz2005,Maz2005b} for a proof of the following results.

\begin{prop}\label{prop:ECP_bernstein-like}
Given a piecewise Chebyshevian spline space $S(\bUU_m,\bDelta,\bR)$, with knots of zero multiplicity, which contains constants, the following properties are equivalent:
\begin{enumerate}
\item[i)] $S(\bUU_m,\bDelta,\bR)$ is an ECP-space on $I$;
\item[ii)] $S(\bUU_m,\bDelta,\bR)$ possesses a Bernstein-like basis on any $[c,d] \subseteq I$.
\end{enumerate}
\end{prop}

\begin{prop}\label{prop:ECP_Bernstein}
Given a piecewise Chebyshevian spline space $S(\bUU_m,\bDelta,\bR)$, with knots of zero multiplicity, which contains constants, the following properties are equivalent:
\begin{enumerate}
\item[i)] $DS(\bUU_m,\bDelta,\bR)$ is an $(m-1)$-dimensional ECP-space on $I$;
\item[ii)] $S(\bUU_m,\bDelta,\bR)$ possesses the Bernstein basis on any $[c,d] \subseteq I$.
\end{enumerate}
\end{prop}


In characterizing when a spline space is suitable for design purposes, the concept of knot insertion plays a key role.
A spline space $\hat{S}$ is said to be obtained from a spline space $S$ by knot insertion whenever $\hat{S}$ and $S$ have section spaces of the same dimension and $S\subset \hat{S}$.
In particular, when a new knot is inserted so as to increase the multiplicity of an existing knot, the related connection matrix must be updated by removing its last row and column. When a new knot is inserted in a location that does not correspond to any already existing knot, then the corresponding connection matrix must be the identity matrix.

The following proposition shows the strong link between ECP-spaces and spline spaces that are suitable for design.

\begin{prop}[Theorem 3.2 in \cite{Maz2011b}]
\label{prop:Mazure}
Given a piecewise Chebyshevian spline space $S(\bUU_m,\bDelta,\bR)$, with knots of zero multiplicity, which contains constants, the following properties are equivalent:
\begin{itemize}
\item[i)]\label{prop:Mazure_i} the space $DS(\bUU_m,\bDelta,\bR)$ is an ECP-space on $I$;
\item[ii)] \label{prop:Mazure_ii}$S(\bUU_m,\bDelta,\bR)$ is ``good for design" (meaning existence of blossoms);
\item[iii)] \label{prop:Mazure_iii}for any $[c,d]\subseteq [a,b]$, there exists a Bernstein basis in the restriction of $S(\bUU_m,\bDelta,\bR)$ to $[c,d]$;
\item[iv)] \label{prop:Mazure_iv}any spline space $S(\bUU_m,\bDelta,\bM,\bR)$ based on $S(\bUU_m,\bDelta,\bR)$ is ``good for design".
\end{itemize}
\end{prop}

A consequence of the latter proposition is that, if any of the Properties \emph{i)--iv)} holds, then
the Bernstein basis with respect to any subinterval $[c,d]\subseteq[a,b]$ is the ONTP basis in the restriction of $S(\bUU_m,\bDelta,\bR)$ to such interval.
Moreover, the B-spline basis in any spline space mentioned in \emph{iv)} does exist and is the ONTP basis.

The above proposition and the results previously recalled allow us to say that, whenever it is possible to find a system of piecewise weight functions associated with the space $DS(\bUU_m,\bDelta,\bR)$,
both $S(\bUU_m,\bDelta,\bR)$ and $DS(\bUU_m,\bDelta,\bR)$ are ECP-spaces and thus
both $S(\bUU_m,\bDelta,\bR)$ and any spline space with knots of arbitrary multiplicity based on it are suitable for design.
More precisely, if $DS(\bUU_m,\bDelta,\bR)= ECP(w_1,\dots,w_{m-1})$, then $S(\bUU_m,\bDelta,\bR) = ECP(\mathbbm{1},w_1,\dots,w_{m-1})$.

In the next section we propose a practical method to construct the weight functions associated with a given spline space with knots of zero multiplicity and to determine, depending on their existence, whether the space in question is an ECP-space.

\section{A simple process to construct the weight functions}
\label{sec:suff_cond}

Referring to the same setting and notation introduced in the previous section,
in the remainder of the paper we assume that,
for all $i=0,\dots,k$, $\UU_{i,m}$ is an $m$-dimensional EC-space which contains constants and that
$D\UU_{i,m}$ is an $(m-1)$-dimensional EC-space.
By the non-piecewise version of Proposition \ref{prop:ECP_Bernstein} \cite{CMP2003},
these requirements are equivalent to the existence of the Bernstein basis of each $\UU_{i,m}$ on $I_i$.
In addition, it can be seen from Proposition \ref{prop:Mazure} that they are necessary conditions for constructing spline spaces suitable for design.

\begin{defn}[Transition functions]\label{def:trans_function}
Let $S(\bUU_m,\bDelta,\bR)$ be a piecewise Chebyshevian spline space containing constants.
We call \emph{transition functions relative to $[a,b]$} the functions $f_{\ell,m} \in S(\bUU_m,\bDelta,\bR)$, such that
$f_{1,m}\equiv 1$ and $f_{\ell,m}$, $\ell=2,\dots,m$,
satisfies
\begin{equation}\label{eq:cond_f}
\begin{alignedat}{3}
& D^r f_{\ell,m}(a)=0, &\qquad& r=0,\dots,\ell-2,\\
& D^r f_{\ell,m}(b)= \delta_{r,0}, &\qquad& r=0,\dots,m-\ell.
\end{alignedat}
\end{equation}
Furthermore, for any interval $[c,d]\subset [a,b]$, we call transition functions relative to $[c,d]$
the functions $f_{\ell,m}$ in the restriction of $S(\bUU_m,\bDelta,\bR)$ to $[c,d]$ that satisfy \eqref{eq:cond_f} at
$c$ and $d$.
\end{defn}



\smallskip
For a given spline space $S(\bUU_m,\bDelta,\bR)$,
each transition function $f_{\ell,m}$, $\ell=2,\dots,m$ relative to $[a,b]$ can be determined as the solution of a suitable linear system.
In particular, let $\UU_{i,m}$ be the space spanned by the functions $\{u^{[i]}_{1,m},u^{[i]}_{2,m},\dots,u^{[i]}_{m,m}\}$, with $u^{[i]}_{1,m}=\mathbbm{1}$, and let $f_{\ell,m}^{[i]}\in \UU_{i,m}$ be the restriction of $f_{\ell,m}$ to the interval $I_i$, $i=0,\dots,k$.
Therefore, there will be coefficients such that $f_{\ell,m}^{[i]}(x)=\sum\nolimits_{h=1}^{m}b^{[i]}_{h,\ell,m} u^{[i]}_{h,m}(x)$, $x\in [x_i,x_{i+1}]$.
By imposing conditions \eqref{eq:cond_f} at $a$ and $b$ and by requiring that, for $i=1,\dots,k$,
\begin{equation}
\left(D^0 f_{\ell,m}^{[i]}(x_i),D^1 f_{\ell,m}^{[i]}(x_i),\dots,D^{(m-1)} f_{\ell,m}^{[i]}(x_i)\right)^T
= R_i
\left(D^0 f_{\ell,m}^{[i-1]}(x_i),D^1 f_{\ell,m}^{[i-1]}(x_i),\dots,D^{(m-1)} f_{\ell,m}^{[i-1]}(x_i)\right)^T, 
\end{equation}
we get the linear system
\begin{equation}\label{eq:linear_system}
\matr{A}\vect{b}=\vect{c},
\end{equation}
with
\[
\matr{A} \coloneqq
\begin{pmatrix}
\hspace{2.1ex}\tilde{\matr{A}}_0(x_{0}) & & & & \\
R_1\matr{A}_0(x_{1}) & \hspace{0.8ex}-\matr{A}_1(x_{1}) & & & \\
 & \hspace{0ex} R_2\matr{A}_1(x_{2}) & -\matr{A}_2(x_{2}) & & \\
 & & \ddots & \ddots & \\
 & & & R_k\matr{A}_{k-1}(x_{k}) & -\matr{A}_{k}(x_{k}) \\
 & & & & \hspace{3ex} \tilde{\matr{A}}_{k}(x_{k+1}) \\
\end{pmatrix},
\]
\[
\vect{b} \coloneqq (b_{1,\ell,m}^{[0]},\dots,b_{m,\ell,m}^{[0]},\dots,b_{1,\ell,m}^{[k]},\dots,b_{m,\ell,m}^{[k]})^T,
\qquad
\vect{c} \coloneqq (0,\dots,0,1,\underbrace{0,\dots,0}_{m-\ell\text{ times}})^T.
\]

For all $i=1,\dots,k$, the blocks $\matr{A}_{h}(x_i)$, $h=i-1,i$,
have $r$th row equal to
$\left(D^{r-1}u^{[h]}_{1,m}(x_{i}), \dots, D^{r-1}u^{[h]}_{m,m}(x_{i}) \right)$, $r=1,\dots,m$.
The two blocks $\tilde{\matr{A}}_0(x_{0})$ and $\tilde{\matr{A}}_{k}(x_{k+1})$ are sub-matrices of $A_0(x_0)$ and $A_k(x_{k+1})$ of dimension $(\ell-1)\times m$ and $(m-\ell+1)\times m$ respectively.

By definition each transition function is determined as the solution of an Hermite interpolation problem in $m$-data in the $m$-dimensional spline space $S(\bUU_m,\bDelta,\bR)$. Therefore, it is always possible to find a set of transition functions $f_{\ell,m}$, $\ell = 2,\dots,m$ when $S(\bUU_m,\bDelta,\bR)$ is an ECP-space.
When $DS(\bUU_m,\bDelta,\bR)$ is an ECP-space too, then each $f_{\ell,m}$, $\ell=2,\dots,m$ vanishes at $a$ or $b$ exactly as many times as required by Definition \ref{def:trans_function} and therefore the transition functions are linearly independent.
On the other hand, it shall be noted that the systems \eqref{eq:linear_system} may have a unique solution, and thus we may be able to compute all the transition functions, also when $S(\bUU_m,\bDelta,\bR)$ is not an ECP-space. Moreover, also when $DS(\bUU_m,\bDelta,\bR)$ is not an ECP-space, we may find a set of linearly independent transition functions.

The following characterization holds.

\begin{prop}\label{prop:fmc}
Let $S(\bUU_m,\bDelta,\bR)$ be a piecewise Chebyshevian spline space containing constants and suppose that $DS(\bUU_m,\bDelta,\bR)$ is an ECP-space. Then the transition functions $f_{\ell,m}$, $\ell=2,\dots,m$ relative to any $[c,d]\subseteq[a,b]$ are monotonically increasing, and the set $\{Df_{\ell,m}, \ell=2,\dots,m\}$ is a Bernstein-like basis in the restriction of $DS(\bUU_m,\bDelta,\bR)$ to $[c,d]$.
\end{prop}

\begin{proof}
Since $DS(\bUU_m,\bDelta,\bR)$ is an $(m-1)$-dimensional ECP-space, any nonzero function contained in it can have at most $m-2$ zeros.
In particular, for any $\ell=2,\dots,m$, $Df_{\ell,m}$ belongs to $DS(\bUU_m,\bDelta,\bR)$, vanishes $\ell-2$ times at $c$ and $m-\ell$ times at $d$, and therefore cannot be zero anywhere else in $[c,d]$. Moreover, since $f_{\ell,m}$ takes the values $0$ and $1$ respectively at $c$ and $d$, $Df_{\ell,m}$ is positive in $(c,d)$. There follows that the functions $Df_{\ell,m}$, $\ell=2,\dots,m$ form a Bernstein-like basis in the restriction of $DS(\bUU_m,\bDelta,\bR)$ to $[c,d]$.
\end{proof}

Under the hypotheses of the above proposition, it can be verified that
the set $B_{\ell,m}=f_{\ell,m}-f_{\ell+1,m}$, $\ell=1,\dots,m-1$, $B_{m,m}=f_{m,m}$ is a Bernstein basis in the restriction of $S(\bUU_m,\bDelta,\bR)$ to any $[c,d]\subseteq [a,b]$, namely it fulfills the properties in Definition \ref{def:Bernst_basis}. As a consequence, such basis is the ONTP basis.


For a given space $S(\bUU_m,\bDelta,\bR)$, under the assumption that the transition functions $f_{\ell,m}$, $\ell=1,\dots,m$, relative to $[a,b]$, are linearly independent, we consider the sequence of functions $w_j$, $j=0,\dots,m-1$, where $w_0 = 1$ and, for all $j=1,\dots,m-1$, $w_j$ is constructed recursively by the formula
\begin{equation}\label{eq:wi}
w_{j} = \sum_{\ell=2}^{m-j+1} D f_{\ell,m-j+1},
\end{equation}
\begin{equation}\label{eq:fimj}
f_{\ell,m-j}=\frac{\sum_{h = \ell+1}^{m-j+1} Df_{h,m-j+1}}{w_{j}}, \quad \ell = 1,\dots,m-j.
\end{equation}

If the functions $w_j$ generated by \eqref{eq:wi} are positive for all $j=1,\dots,m-1$, then the set $\{f_{\ell,m-j}, \ell=1,\dots,m-j\}$ computed through \eqref{eq:fimj} is well-defined and made of linearly independent functions.
In particular $\{f_{\ell,m-j}, \ell=1,\dots,m-j\}$ are the transition functions for the space $L_jS(\bUU_m,\bDelta,\bR)\coloneqq \{L_jF \, | \, F \in S(\bUU_m,\bDelta,\bR)\}$.

It can also be verified that,
for any $F\in S(\bUU_m,\bDelta,\bR)$, the functions $w_j$ generated by \eqref{eq:wi} are such that $L_jF^{[i-1]}(x_i)=L_jF^{[i]}(x_i)$, for all $i=1,\dots,k$. 
Moreover, by applying \eqref{eq:fimj} with $j=m-1$, it can be seen that $L_{m-1}S(\bUU_m,\bDelta,\bR)$ is the space spanned by $f_{1,1}\equiv1$.
There follows that $S(\bUU_m,\bDelta,\bR)$ is an ECP-space (see Proposition \ref{prop:ECP_def}).
In addition, all the intermediate spaces $L_jS(\bUU_m,\bDelta,\bR)$, $j=1,\dots,m-2$ are ECP-spaces and in particular $DS(\bUU_m,\bDelta,\bR)=w_1 L_1S(\bUU_m,\bDelta,\bR)$ is an ECP-space too.
The above observations can be summarized in the following proposition.

\begin{prop}\label{prop:wi_positive_ECP}
Let $S(\bUU_m,\bDelta,\bR)$ be a spline space containing constants and suppose that all the functions $w_j$ $j=1,\dots,m-1$, generated by formula \eqref{eq:wi} are positive. Then $S(\bUU_m,\bDelta,\bR)$ is an ECP-space associated with the system of piecewise weight functions $\{\mathbbm{1},w_1,\dots,w_{m-1}\}$.
Moreover, $DS(\bUU_m,\bDelta,\bR)$ is an ECP-space associated with $\{w_1,\dots,w_{m-1}\}$.
\end{prop}


\begin{rem}
Another way to see that $S(\bUU_m,\bDelta,\bR)$ is the ECP-space associated with the functions $w_j$, $j=1,\dots,m-1$ in \eqref{eq:wi}, subject to their positivity, is to consider the set of functions
\begin{equation}\label{eq:canonical_basis}
\begin{aligned}
\psi_1(x) &= w_0(x)\\
\psi_2(x) &= w_0(x) \int_a^{x}w_1(\xi_1) d\xi_1,\\
\psi_{r+1}(x) &= w_0(x) \int_a^{x}w_1(\xi_1) \int_a^{\xi_1} \dots \int_a^{\xi_{r-1}}w_{r}(\xi_r) d\xi_r \dots d\xi_2 d\xi_1, \qquad r=2,\dots,m-1.
\end{aligned}
\end{equation}
By substituting \eqref{eq:wi} and \eqref{eq:fimj} into the above expressions and recalling that, by definition, $f_{\ell,m-j}(a)=0$, $\ell=2,\dots,m-j$, we get
\begin{equation}\label{eq:canonical_basis_f}
\psi_1(x) = 1, \qquad
\psi_2(x) = \sum_{\ell=2}^m f_{\ell,m}(x), \qquad
\psi_{r+1}(x) = \sum_{\ell=r+1}^m \binom{\ell-2}{r-1} f_{\ell,m}(x), \quad r=2,\dots,m-1,
\end{equation}
where, in particular, $\psi_{m}=f_{m,m}$. Therefore $\{\psi_j, j=1,\dots,m\}$ is a canonical basis for $S(\bUU_m,\bDelta,\bR)$.
\end{rem}

The following Proposition \ref{prop:ECP_wi_positive} shows that, when the space $DS(\bUU_m,\bDelta,\bR)$ is an ECP-space, then formula \eqref{eq:wi} always yields an associated system of piecewise weight functions. As a consequence, it turns out that, when at least one of the functions $w_j$ is nonpositive, then both $DS(\bUU_m,\bDelta,\bR)$ and $S(\bUU_m,\bDelta,\bR)$ cannot be ECP-spaces.
A preliminary lemma is needed to prove the main result.

\begin{lem}\label{lem:wcoeff}
Let $S(\bUU_m,\bDelta,\bR)$ be a piecewise Chebyshevian spline space containing constants and suppose that $DS(\bUU_m,\bDelta,\bR)$ is an ECP-space. Then any function $w$ which has positive coefficients in a Bernstein-like basis of $DS(\bUU_m,\bDelta,\bR)$ on $[a,b]$ can be represented with positive coefficients in a Bernstein-like basis in the restriction of $DS(\bUU_m,\bDelta,\bR)$ to any $[c,d]\subset [a,b]$.
\end{lem}

\begin{proof}
By Proposition \ref{prop:fmc},
since $w$ has positive coefficients in a Bernstein-like basis of  $DS(\bUU_m,\bDelta,\bR)$, it can be represented as
$w=\sum_{i=1}^{m-1} \alpha_i Df_{i+1,m}$, where $\alpha_i>0$, for all $i=1,\dots,m-1$ and $f_{i,m}$, $i=1,\dots,m$ are the transition functions of $S(\bUU_m,\bDelta,\bR)$ relative to $[a,b]$.  Let $u$ be a function such that $Du=w$. Without loss of generality, we can write
\begin{equation}
u = \sum_{i=1}^m \alpha_i f_{i,m} = \sum_{i=1}^m \alpha_i \sum_{\ell=i}^m B_{\ell,m} = \sum_{i=1}^m \sum_{\ell=1}^i \alpha_\ell  B_{i,m},
\end{equation}
where $\{B_{i,m},\, i=1,\dots,m\}$ is the Bernstein basis of $S(\bUU_m,\bDelta,\bR)$.
There follows that, in such basis, $u$ has increasing coefficients.

As a consequence of Proposition \ref{prop:Mazure}, for any $[c,d]\subseteq [a,b]$ the Bernstein basis in the restriction of $S(\bUU_m,\bDelta,\bR)$ to $[c,d]$ is an ONTP basis. Thus, the function $u$ has increasing coefficients in the Bernstein basis $\{\hat{B}_{i,m}, \, i=1,\dots,m\}$ in the restriction of $S(\bUU_m,\bDelta,\bR)$ to any $[c,d]\subset [a,b]$. In particular, let
\begin{equation}
\restr{u}{[c,d]}= \sum_{i=1}^m \gamma_i \hat{B}_{i,m}, \quad \gamma_{i+1}-\gamma_i>0, \; \text{for all } i,
\end{equation}
where $\hat{B}_{i,m}= \hat{f}_{i,m}-\hat{f}_{i+1,m}$, $i=1,\dots,m-1$, $\hat{B}_{m,m}=\hat{f}_{m,m}$ and $\hat{f}_{i,m}$ are the transition functions relative to $[c,d]$.
By Proposition \ref{prop:fmc}, the transition functions are monotonically increasing and their derivatives are a Bernstein-like basis in the restriction of $DS(\bUU_m,\bDelta,\bR)$ to $[c,d]$. The statement then follows by observing that
\begin{equation}
\restr{Du}{[c,d]}= \restr{w}{[c,d]}= \sum_{i=1}^{m-1} (\gamma_{i+1}-\gamma_{i}) D\hat{f}_{i+1,m}.
\end{equation}
\end{proof}

\begin{prop}\label{prop:ECP_wi_positive}
Let $S(\bUU_m,\bDelta,\bR)$ be a piecewise Chebyshevian spline space containing constants and suppose that $DS(\bUU_m,\bDelta,\bR)$ is an ECP-space. Then the sequence of functions $\{\mathbbm{1},w_1,\dots,w_{m-1}\}$ determined by \eqref{eq:wi}--\eqref{eq:fimj} is a system of piecewise weight functions associated with $S(\bUU_m,\bDelta,\bR)$. Moreover the sequence $\{w_1,\dots,w_{m-1}\}$ is a system of piecewise weight functions associated with $DS(\bUU_m,\bDelta,\bR)$.
\end{prop}

\begin{proof}
The functions $w_j$, $j=0,\dots,m-1$, are by construction piecewise $C^{m-j-1}$ on $(I,\bDelta)$.
We shall then prove that they are positive.
Let $f_{\ell,m}$, $\ell=1,\dots,m$ be the transition functions of $S(\bUU_m,\bDelta,\bR)$ relative to $[a,b]$.
From Proposition \ref{prop:fmc}, $Df_{\ell,m}$, $\ell=2,\dots,m$, is a Bernstein-like basis for
$DS(\bUU_m,\bDelta,\bR)$ and therefore $w_1$ in \eqref{eq:wi} is positive on $I$ and
the sequence $\{f_{\ell,m-1}, \ell=1,\dots,m-1\}$ defined by \eqref{eq:fimj} is a Bernstein basis for $L_1S(\bUU_m,\bDelta,\bR)$.

Being $DS(\bUU_m,\bDelta,\bR)$ an ECP-space on $[a,b]$, there exists a Bernstein-like basis in the restriction of $DS(\bUU_m,\bDelta,\bR)$ to any $[c,d]\subset [a,b]$ and, by Lemma \ref{lem:wcoeff},
$w_1$ can be represented with positive coefficients in such basis.
Hence, if we normalize this Bernstein-like basis by $w_1$, we get a Bernstein basis in the restriction of $L_1S(\bUU_m,\bDelta,\bR)$ to $[c,d]$. This shows that there exists a Bernstein basis in the restriction of $L_1S(\bUU_m,\bDelta,\bR)$ to any $[c,d]\subseteq [a,b]$ and thus,
by Proposition \ref{prop:ECP_Bernstein}, $DL_1S(\bUU_m,\bDelta,\bR)$ is an ECP-space on $[a,b]$. As a consequence (see Proposition \ref{prop:fmc}) the functions $f_{\ell,m-1}$, $\ell=2,\dots,m-1$, are monotonically increasing and hence $w_2>0$ on $I$. The positivity of $w_2$ allows us to repeat the above reasoning to conclude that $L_2S(\bUU_m,\bDelta,\bR)$ is an ECP-space.

By applying the same argument iteratively, it can be proven that all the spaces $DL_jS(\bUU_m,\bDelta,\bR)$ and $L_jS(\bUU_m,\bDelta,\bR)$ generated by the procedure \eqref{eq:wi}--\eqref{eq:fimj} are ECP-spaces and that the functions $w_j$ are positive for all $j=1,\dots,m-1$.
\end{proof}

\section{A numerical procedure to determine if a spline space with knots of zero multiplicity is suitable for design}
\label{sec:test}
Let $S(\bUU_m,\bDelta,\bR)$ be a given spline space where the underlying local EC-spaces $\UU_{i,m}$ contain constants and where $D\UU_{i,m}$ is an EC-space on $I_i$, for all $i=0,\dots,k$.
Our objective is to exploit the results presented in the previous section to develop a numerical procedure for determining whether the considered space is suitable for design.

The first step of the procedure consists in computing the transition functions relative to $[a,b]$. At this stage, if any of the systems \eqref{eq:linear_system} does not have a unique solution, then $S(\bUU_m,\bDelta,\bR)$ is not an ECP-space and therefore it cannot be suitable for design.

Successively, we need to verify that the transition functions are linearly independent.
When this is not the case, we can directly conclude that $DS(\bUU_m,\bDelta,\bR)$ is not an ECP-space (see Proposition \ref{prop:fmc}), and it is unnecessary to carry out the next steps.
The linear independence of the transition functions can be easily assessed by checking that
each of them vanishes at $a$ (or equivalently at $b$) \emph{exactly} as many times as required by definition.
Therefore, when the transition functions are linearly independent, we must have
$D^{(\ell-1)}f_{\ell,m}(a)\neq 0$, and $D^{(m-\ell+1)}(1-f_{\ell,m})(b)\neq 0$, for all $\ell=2,\dots,m$.
Conversely, if there is an integer $r\in \{2,\dots,m\}$, such that $f_{r,m}$ vanishes more than $r-1$ times at $a$ (or $1-f_{r,m}$ vanishes more than $m-r+1$ times at $b$),
then the transition functions are linearly dependent.

Hence, we shall proceed under the assumption that the transition functions $f_{\ell,m}$, $\ell=1,\dots,m$ exist and are linearly independent.
In the remaining part of the section we discuss how to determine whether or not the space $S(\bUU_m,\bDelta,\bR)$ is suitable for design in a computationally efficient way.

Let $L_j\UU_{i,m}$, $j=0,\dots,m-1$, be the space
obtained by generalized differentiation restricted to the interval $I_i$ and let $\{B_{\ell,m-j}^{[i]}, \ell=1,\dots,m-j\}$ be its Bernstein basis on the same interval.

For any $j=1,\dots,m-1$, we represent the $i$th piece of $f_{\ell,m-j+1}$ in the (local) Bernstein basis of $L_{j-1}\UU_{i,m}$ on $I_i$, $i=0,\dots,k$ as
\begin{equation}\label{eq:fiBernst}
f^{[i]}_{\ell,m-j+1}=\sum_{h=1}^m b^{[i]}_{h,\ell,m-j+1}B^{[i]}_{h,m-j+1}, \quad \ell=1,\dots,m-j+1.
\end{equation}

Our working assumptions (namely that $\UU_{i,m}$ contains constants and that $\UU_{i,m}$ and $D\UU_{i,m}$ are EC-spaces) guarantee the existence of the Bernstein basis for $\UU_{i,m}=L_0\UU_{i,m}$ on $I_i$.
Such basis is given by $\{B^{[i]}_{\ell,m}=g^{[i]}_{\ell,m}-g^{[i]}_{\ell+1,m}, \; \ell=1,\dots,m-1,\; B^{[i]}_{m,m}=g^{[i]}_{m,m}\}$, where $g^{[i]}_{\ell,m}$, $\ell=1,\dots,m$, are the transition functions for the considered space relative to $I_i$. The latter can be computed by imposing the conditions \eqref{eq:cond_f} at $x_i$ and $x_{i+1}$.
For any other space $L_{j-1}\UU_{i,m}$, $j=2,\dots,m-1$ the Bernstein basis on $I_i$, whenever it exists, can be computed recursively, as we will see in the following.

Now, for any $j=1,\dots,m-1$, set
\begin{equation}
\tilde{f}_{\ell,m-j}\coloneqq \sum_{r=\ell+1}^{m-j+1} Df_{r,m-j+1}, \qquad \ell=1,\dots,m-j,
\end{equation}
in such a way that the function $w_j$ defined in \eqref{eq:wi} is precisely equal to $\tilde{f}_{1,m-j}$.
In $I_i$, $i=0,\dots,k$, we can
substitute \eqref{eq:fiBernst} into the above equation, obtaining
\begin{equation}
\tilde{f}^{[i]}_{\ell,m-j}
= \sum_{r=\ell+1}^{m-j+1} \sum_{h=1}^{m-j+1} b^{[i]}_{h,r,m-j+1}D B^{[i]}_{h,m-j+1}
= \sum_{r=\ell+1}^{m-j+1} \sum_{h=1}^{m-j+1} b^{[i]}_{h,r,m-j+1}\left(D g^{[i]}_{h,m-j+1}-Dg^{[i]}_{h+1,m-j+1}\right),
\end{equation}
where $g^{[i]}_{\ell,m-j+1}$, $\ell=1,\dots,m-j+1$, are the transition functions of $L_{j-1}\UU_{i,m}$ relative to $I_i$.
Recalling that $g^{[i]}_{1,m-j+1}\equiv 1$ and $g^{[i]}_{m-j+2,m-j+1}\equiv 0$ for all $j$, we get
\begin{align}
\tilde{f}^{[i]}_{\ell,m-j}
&= \sum_{h=1}^{m-j}  \sum_{r=\ell+1}^{m-j+1} \left(b^{[i]}_{h+1,r,m-j+1}-b^{[i]}_{h,r,m-j+1}\right) Dg^{[i]}_{h+1,m-j+1}\\[1ex]
&= \sum_{h=1}^{m-j} \tilde{b}^{[i]}_{h,\ell,m-j} D g^{[i]}_{h+1,m-j+1},\label{eq:btilde}
\end{align}
where we have set
\begin{equation}\label{eq:Bernst_like_new}
\tilde{b}^{[i]}_{h,\ell,m-j}\coloneqq\sum_{r=\ell+1}^{m-j+1} \left(b^{[i]}_{h+1,r,m-j+1}-b^{[i]}_{h,r,m-j+1}\right), \qquad h = 1,\dots,m-j.
\end{equation}

\noindent
From \eqref{eq:fimj} and \eqref{eq:btilde} we get
\begin{equation}\label{eq:fmmjBernst}
f^{[i]}_{\ell,m-j}=\frac{\tilde{f}^{[i]}_{\ell,m-j}}{w_j}=
\frac{\sum_{h=1}^{m-j} \tilde{b}^{[i]}_{h,\ell,m-j} D g^{[i]}_{h+1,m-j+1}}
{\sum_{h=1}^{m-j} \tilde{b}^{[i]}_{h,1,m-j} D g^{[i]}_{h+1,m-j+1}}=
\sum_{h=1}^{m-j} \frac{\tilde{b}^{[i]}_{h,\ell,m-j}}{\tilde{b}^{[i]}_{h,1,m-j}} B^{[i]}_{h,m-j}=
\sum_{h=1}^{m-j} b^{[i]}_{h,\ell,m-j} B^{[i]}_{h,m-j},
\end{equation}
where
\begin{equation}\label{eq:Bernst_new}
B^{[i]}_{h,m-j}\coloneqq
\frac{\tilde{b}^{[i]}_{h,1,m-j} D g^{[i]}_{h+1,m-j+1}}
{\sum_{r=1}^{m-j} \tilde{b}^{[i]}_{r,1,m-j} D g^{[i]}_{r+1,m-j+1}}, \qquad h=1,\dots,m-j,
\end{equation}
\noindent and
\begin{equation}\label{eq:coeffsBernst_new}
b^{[i]}_{h,\ell,m-j}\coloneqq \frac{\tilde{b}^{[i]}_{h,\ell,m-j}}{\tilde{b}^{[i]}_{h,1,m-j}}.
\end{equation}

\noindent
Note that the denominator in \eqref{eq:Bernst_new} is precisely the function $w_j$ restricted to $I_i$.

To verify if the space $S(\bUU_m,\bDelta,\bR)$ is suitable for design, we can proceed iteratively.
For each step $j=1,\dots,m-1$, we test if
the differences $b^{[i]}_{h+1,r,m-j+1}-b^{[i]}_{h,r,m-j+1}$ in equation \eqref{eq:Bernst_like_new} are nonnegative for all $h=1,\dots,m-j$, $r=2,\dots,m-j+1$.
If the test is successful, all the transition functions $f_{\ell,m-j+1}$, $\ell=1,\dots,m-j+1$ have non-decreasing coefficients in the Bernstein bases of $L_{j-1}\UU_{i,m}$ on $I_i$, for $i=0,\dots,k$.
As a consequence, the function $w_{j}=\sum_{\ell=2}^{m-j+1} Df_{\ell,m-j+1}$ is positive on the whole of $I$ and we can proceed to the successive step.
In addition, the success of the test at step $j$ implies that all the coefficients $\tilde{b}_{h,\ell,m-j}^{[i]}$ on the left-hand side of equation \eqref{eq:Bernst_like_new} are positive and therefore
equation \eqref{eq:Bernst_new} yields the Bernstein basis of $L_{j}\UU_{i,m}$ on $I_i$.

Conversely, if at any step $j$ the test fails, this means that there is at least one $\ell=2,\dots,m-j+1$ and one interval $I_i$ such that the coefficients of $f_{\ell,m-j+1}^{[i]}$ in the Bernstein basis of $L_{j-1}\UU_{i,m}$ on $I_i$ do not form a non-decreasing sequence.
In this case, we can immediately conclude that the space $S(\bUU_m,\bDelta,\bR)$ is not suitable for design and stop the testing procedure.

To explain the latter statement, we shall recall that, if $S(\bUU_m,\bDelta,\bR)$ is suitable for design, then so is each space $L_{j-1} S(\bUU_m,\bDelta,\bR)$, $j=2,\dots,m-1$ generated by an associated system of piecewise weight functions.
In any of these spaces, by Proposition \ref{prop:fmc}, the transition functions $f_{\ell,m-j+1}$, $\ell=1,\dots,m-j+1$ relative to $[a,b]$ are monotonically increasing and their coefficients in the related Bernstein basis $\{B_{\ell,m-j+1}=f_{\ell,m-j+1}-f_{\ell+1,m-j+1}, \, \ell=1,\dots,m-j-1+1,\, B_{m-j+1,m-j+1}=f_{m-j+1,m-j+1}\}$ are non-decreasing. Moreover, the latter property holds considering the Bernstein basis in the restriction of $L_{j-1} S(\bUU_m,\bDelta,\bR)$ to any $I_i$, $i=0,\dots,k$.
Therefore, when the Bernstein coefficients violate this property, nor $L_{j-1}S(\bUU_m,\bDelta,\bR)$ or $S(\bUU_m,\bDelta,\bR)$ can be suitable for design.
It is interesting to note that, in this way, even though $w_{j}$ could be positive, the numerical procedure anticipates the non-positivity of one of the successive functions $w_h$, $h>j$.

From the computational point of view, the transition functions can be efficiently computed in an iterative way, where at each step $j=1,\dots,m-1$ the functions $f_{\ell,m-j}$, $\ell=1,\dots,m-j$ for $L_j S(\bUU_m,\bDelta,\bR)$ are generated.
In particular, equations \eqref{eq:Bernst_like_new}, \eqref{eq:fmmjBernst} and \eqref{eq:coeffsBernst_new} show that, for any $j=1,\dots,m-1$, the coefficients of $f^{[i]}_{\ell,m-j}$ in the Bernstein basis of $L_j\UU_{i,m}$ on $I_i$ can be recursively computed from the coefficients of the transition functions $f^{[i]}_{\ell,m-j+1}$, $\ell=1,\dots,m-j+1$, in the Bernstein basis of the previous space $L_{j-1}\UU_{i,m}$ on the same interval.

The following MATLAB function takes as input a matrix \texttt{b} of dimension $m\times m\times(k+1)$, where \texttt{b(l,h,i)}$=b^{[i]}_{h,\ell,m}$ are the coefficients of $f_{\ell,m}^{[i]}$ in equation \eqref{eq:fiBernst}. It returns a variable \texttt{test}, which is equal to zero if at any step the test on the monotonicity of the Bernstein coefficients fails.
In this function the loops in the variables \texttt{l}, \texttt{h}, \texttt{i} iterate respectively over the transition functions, the Bernstein coefficients and the knot intervals.
{\small
\begin{Verbatim}[commandchars=\\\{\}]
function test=SfD_test(b)
[m,m,kp1]=size(b);
test=1;
j=0;
while (j<=m-2 & test)
  mj=m-j;
  for i=1:kp1
% difference of subsequent Bernstein coefficients
    for l=2:mj
        for h=1:mj-1
            b(l,h,i)=b(l,h+1,i)-b(l,h,i);
            if (b(l,h,i)<0) test=0;
                return
            end
        end
    end
% summation step according to formula (\ref{eq:Bernst_like_new})
    for l=mj-1:-1:2
        for h=1:mj-1
            b(l,h,i)=b(l,h,i)+b(l+1,h,i);
        end
    end
% division step according to formula (\ref{eq:coeffsBernst_new})
    for l=2:mj
        b2hi=b(2,h,i);
        for h=1:mj-1
            b(l-1,h,i)=b(l,h,i)/b2hi;
        end
    end
  end
  j=j+1;
end
\end{Verbatim}
}

\noindent
We conclude by presenting two application examples.

\begin{exmp}\label{ex:1}
Let us consider a spline space $S(\bUU_4,\bDelta,\bR)$ based on the sequence of section spaces $\bUU_4=\{\UU_{0,4},\UU_{1,4},\allowbreak \UU_{2,4},\UU_{3,4}\}$, with $\UU_{0,4}=\UU_{2,4}=\Span\{1,x,\cos x, \sin x\}$, $\UU_{1,4}=\UU_{3,4}=\Span\{1,x,\cosh x, \sinh x\}$, knot
partition $\{x_0,x_1,x_2,x_3,x_4\}=\{0,2,4,5,6\}$ and connection matrices

\begin{equation}\label{eq:Rex1}
R_1=R_3=
\begin{pmatrix}
1 & 0 & 0 & 0\\
0 & 2 & 0 & 0\\
0 & 0 & 2 & 0\\
0 & 0 & 1 & 4
\end{pmatrix}, \qquad R_2=\mathbbm{I}_4.
\end{equation}

The necessary condition to guarantee that each space $D\UU_{h,4}$, $h=0,2$ is an EC-space on $[x_h,x_{h+1}]$ is $x_{h+1}-x_h<2\pi$, which is fulfilled by the given knots, while $D\UU_{h,4}$, $h=1,3$ are EC-spaces on $\RR$.
It can be observed that the transition functions $f_{\ell,4-j}$, $\ell=2,\dots,4-j$ are monotonically increasing in all the spaces $L_jS(\bUU_4,\bDelta,\bR)$, $j=0,1,2$, (Figures \ref{fig:ex1_1}--\ref{fig:ex1_3}) and that, accordingly, the functions $w_j=\sum_{\ell=2}^{4-j}Df_{\ell,4-j}$, $j=1,2,3$ are positive (Figure \ref{fig:ex1_4}--\ref{fig:ex1_6}, where the functions $w_j$ are depicted in bold). As a consequence, the considered spline space is suitable for design.

If we now take the same sequence of section spaces $\bUU_4$ and the same connection matrices with knots $\{0,0.5,5.3,\allowbreak 10.1,14.9\}$, we get the situation illustrated in Figure \ref{fig:f1_2}. Also in this case the knot intervals fulfill the aforementioned  necessary condition $x_{h+1}-x_h<2\pi$, $h = 0,2$. It is interesting to observe that the transition functions $f_{\ell,4}$, $\ell=2,3,4$ are monotonically increasing (Figure \ref{fig:ex1_7}), which entails that $w_1$ is positive (Figure \ref{fig:ex1_8}) and that the set $\{B_{\ell,4}=f_{\ell,4}-f_{\ell+1,4}, \, \ell=1,\dots,3, \, B_{4,4}=f_{4,4}\}$ is a Bernstein basis in the sense of Definition \ref{def:Bernst_basis} (Figure \ref{fig:ex1_9}).
Nevertheless the considered space is not suitable for design.
In particular the Bernstein coefficients of $f^{[2]}_{2,4}$ do not form a non-decreasing sequence in $I_2$ and therefore, at this stage, our numerical test stops returning a negative response.
If we were to proceed further, we would find that $w_2$ is nonpositive (Figure \ref{fig:ex1_11}) and that the transition functions $f_{\ell,3}$, $\ell=2,3$, are non-monotonically increasing (Figure \ref{fig:ex1_10}).

\end{exmp}

\begin{exmp}\label{ex:2}
We consider some samples of parametric curves from spline spaces with knots of zero multiplicity, where the connection matrices can be used to obtain tension effects useful in geometric modeling.
All the curves illustrated in Figure \ref{fig:f2} are represented in the Bernstein basis of a spline space $S(\bUU_m,\bDelta,\bR)$, given by the sequence $\{B_{\ell,m}=f_{\ell,m}-f_{\ell+1,m}, \, \ell=1,\dots,m-1, \, B_{m,m}=f_{m,m}\}$.
In the three subfigures the knots are the same $\{x_0,x_1,x_2\}=\{0,1,2\}$, whereas the
local spaces $\bUU_m=\{\UU_{0,m},\UU_{1,m}\}$ have different dimension $m=4,5,6$.
In all the figures, $\UU_{0,m}=\UU_{1,m}$ and there is only one connection matrix at $x_1$, which we indicate by $R_1^{\text{(a)}}$, $R_1^{\text{(b)}}$ and $R_1^{\text{(c)}}$ for the three subfigures.
In particular

\begin{equation}\label{eq:R_ex2}
R_1^{\text{(a)}}=
\begin{pmatrix}
1 & 0 & 0 & 0\\
0 & 1 & 0 & 0\\
0 & 0 & 1 & 0\\
0 & 0 & \beta & 1
\end{pmatrix},
\qquad
R_1^{\text{(b)}}=
\begin{pmatrix}
1 & 0 & 0 & 0 & 0\\
0 & 1 & 0 & 0 &0 \\
0 & 0 & 1 & 0 & 0\\
0 & 0 & 0 & 1 & 0\\
0 & 0 & 0 & \beta & 1
\end{pmatrix},
\qquad
R_1^{\text{(c)}}=
\begin{pmatrix}
1 & 0 & 0 & 0 & 0 & 0\\
0 & 1 & 0 & 0 & 0 & 0 \\
0 & 0 & 1 & 0 & 0 & 0\\
0 & 0 & \beta & 1 & 0 & 0\\
0 & 0 & 0 & 0 & 1 & 0 \\
0 & 0 & 0 & 0 & 0 & 1
\end{pmatrix},
\end{equation}

\noindent
in such a way that the variable parameter $\beta$ in the above matrices influences the value of higher order derivatives.

In Figure \ref{fig:ex2_1} the local section spaces are $\UU_{0,4}=\UU_{1,4}=\Span\{1,x, x^2,x^3\}$ and
the displayed curves correspond to $\beta=-3.9, 0, 10, 100$.
In Figure \ref{fig:ex2_2}, $\UU_{0,5}=\UU_{1,5}=\Span\{1,x,x^2,\cos x,\sin x\}$ and
$\beta=-3.5, 0, 10, 100$ for each curve.
In Figure \ref{fig:ex2_3}, $\UU_{0,6}=\UU_{1,6}=\Span\{1,x,\cos x, \allowbreak \sin x,x \cos x, x \sin x\}$
and the different curves are obtained for $\beta=-6.5, -5, 0, 100$.

The necessary condition that all $D\UU_{i,m}$ be EC-spaces is fulfilled in all the considered examples. In particular, the latter condition holds when $x_{i+1}-x_i$ is smaller than $8.9868189$ for $\UU_{i,5}$ \cite{CMP2003} and when
$x_{i+1}-x_i$ is smaller than $2\pi$ for $\UU_{i,6}$.
According to the proposed numerical test, for all the considered values of $\beta$ the underlying spline spaces are suitable for design.

In these examples, one element of the connection matrix acts as a shape or tension parameter, namely, the higher its value, the closer the curve lies to the control polygon.
In such a situation, it is important to be able to progressively increase or decrease the parameter
while staying in the class of spaces suitable for design.
The numerical procedure presented in this section is well suited to this purpose, since it allows for testing in a computationally efficient way whether a specific value of $\beta$ gives rise to an admissible spline space. This means that the test can be performed while the user interactively modifies the parameter according to the shape to be modelled.

Moreover, our test allowed us to determine experimentally that there is a minimum value $\beta_{\min}$ beyond which the corresponding space is no longer suitable for design. Conversely, for any $\beta>\beta_{\min}$ we obtain a space which is suitable for design.
Figure \ref{fig:f2_2} illustrates the transition functions and the Bernstein basis for the spline space $\bUU_4$ with $\beta=-3.9$, $\bUU_5$ with $\beta=-3.5$ and $\bUU_6$ with $\beta=-6.5$. Despite these values still correspond to admissible spaces, they are close to $\beta_{\min}$. Accordingly, as it can be observed from the figures, the transition functions and the Bernstein basis functions are close to becoming linearly dependent.

\ifelsevier
\section*{Acknowledgements}
The authors would like to acknowledge support from the Italian GNCS-INdAM.
We also thank the anonymous reviewers for their valuable comments and suggestions.
\fi




\ifelsevier
\bibliographystyle{model1-num-names} 
\bibliography{AMC_SMART_paper} 
\fi




\begin{figure}
\centering
\subfigure[$f_{\ell,4}$, $\ell=1,\dots,4$]
{\includegraphics[width=0.95\textwidth/3]{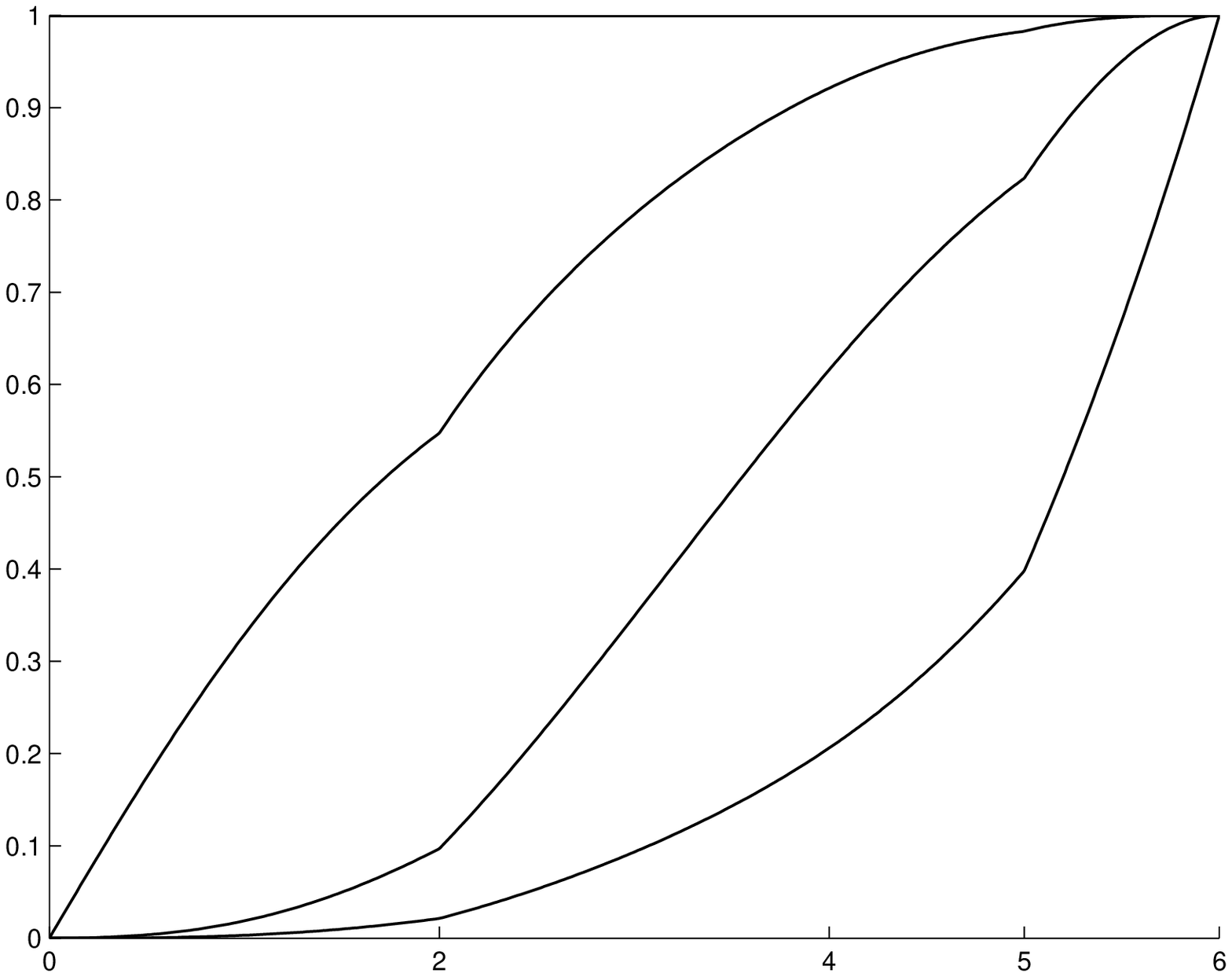}\label{fig:ex1_1}}
\hfill
\subfigure[$f_{\ell,3}$, $\ell=1,2,3$]
{\includegraphics[width=0.95\textwidth/3]{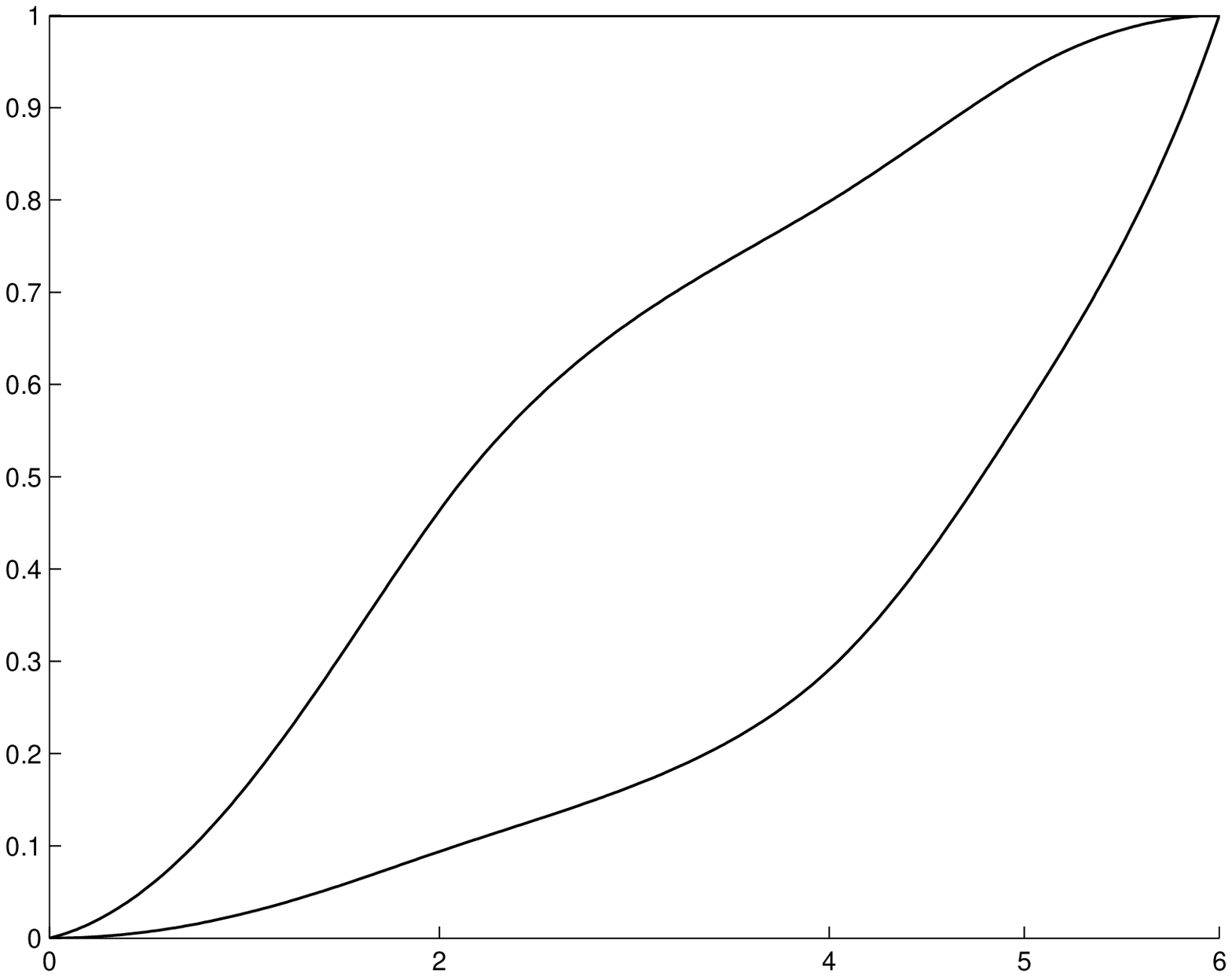}\label{fig:ex1_2}}
\hfill
\subfigure[$f_{\ell,2}$, $\ell=1,2$]
{\includegraphics[width=0.95\textwidth/3]{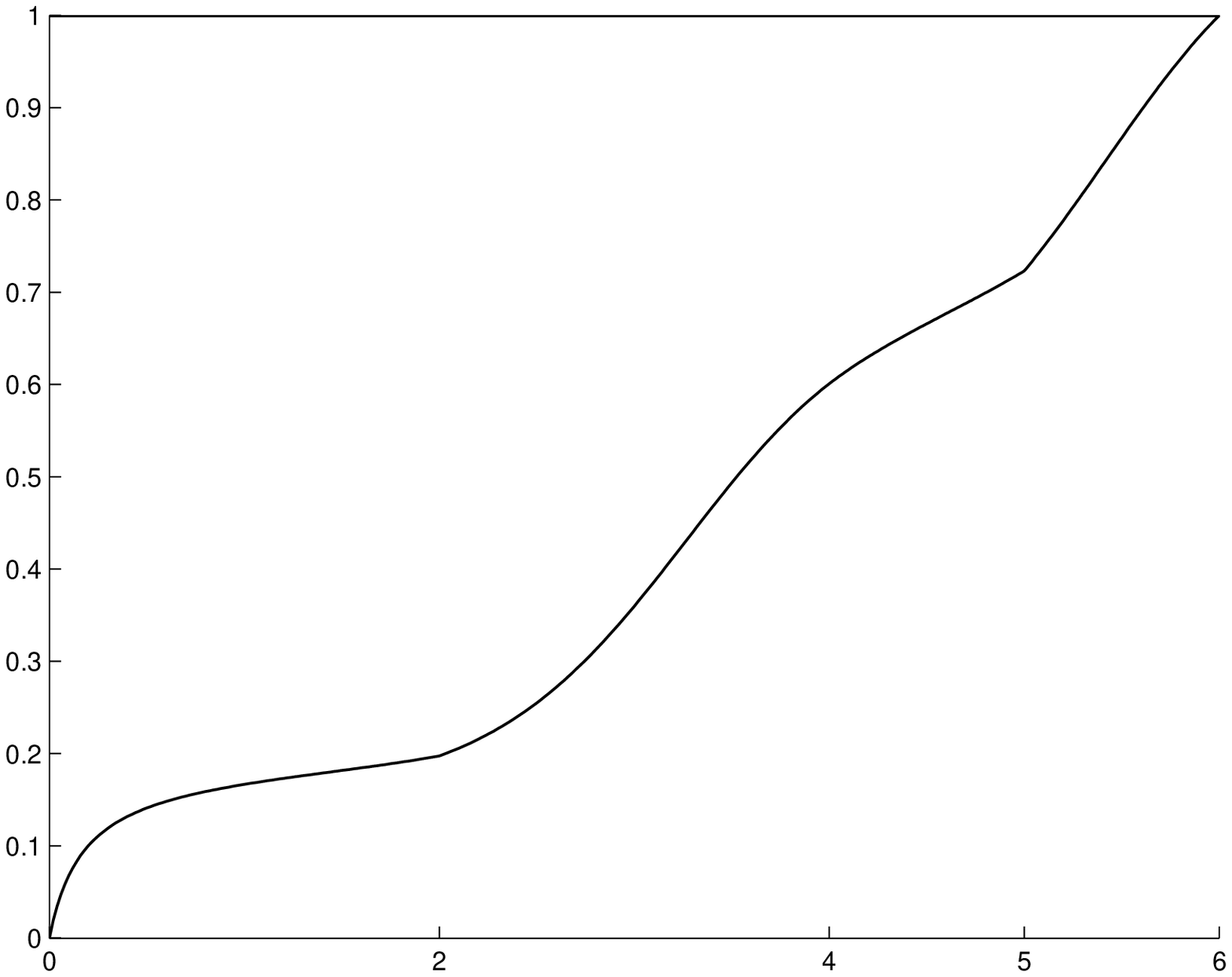}\label{fig:ex1_3}}\\
\subfigure[$\tilde{f}_{\ell,3}=\sum\nolimits_{h=\ell+1}^4 Df_{h,4}$, $\ell=1,2,3$, $\ w_1=\tilde{f}_{1,3}$]
{\includegraphics[width=0.95\textwidth/3]{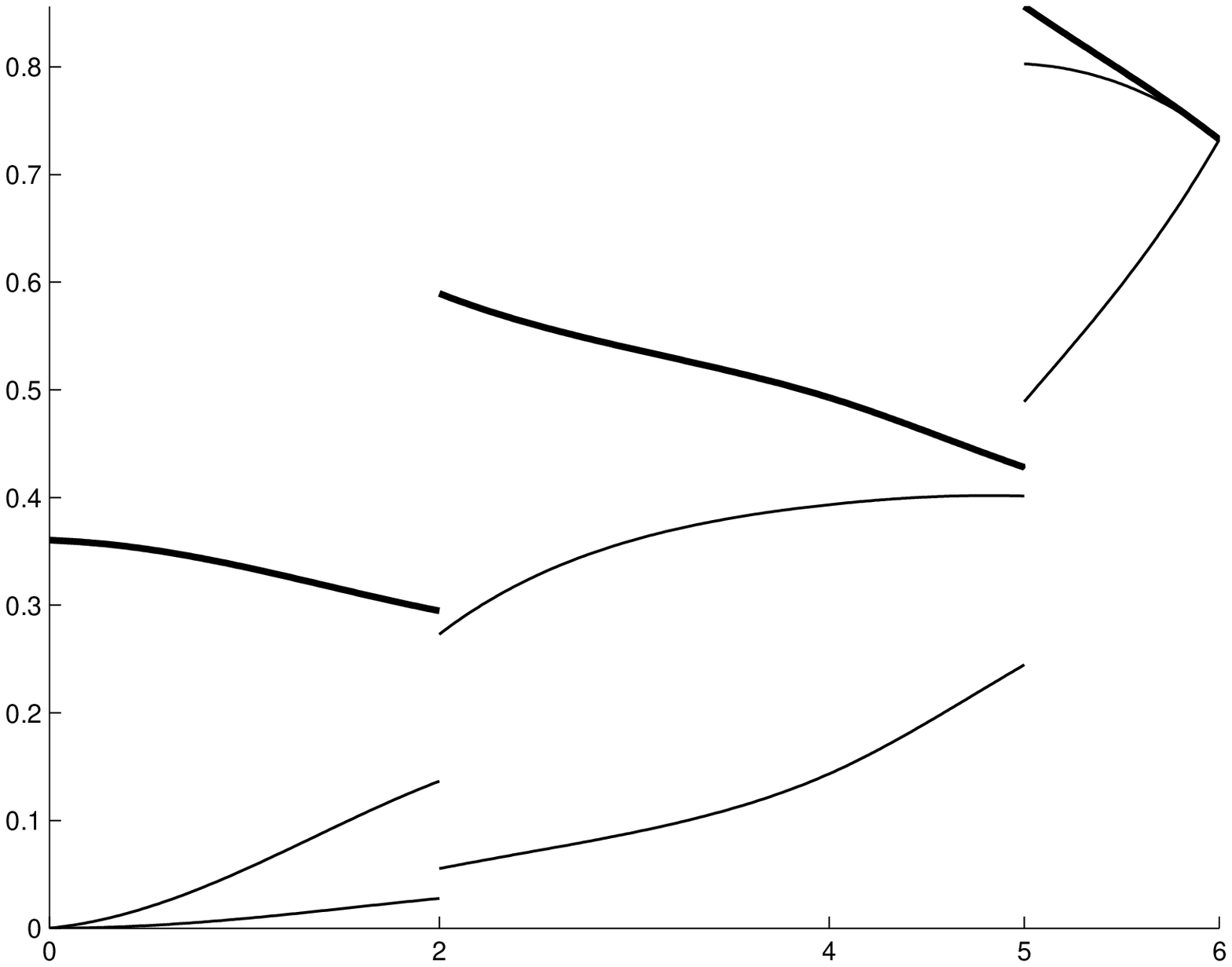}\label{fig:ex1_4}}
\hfill
\subfigure[$\tilde{f}_{\ell,2}=\sum\nolimits_{h=\ell+1}^3 Df_{h,3}$, $\ell=1,2$, $\ w_2=\tilde{f}_{1,2}$]
{\includegraphics[width=0.95\textwidth/3]{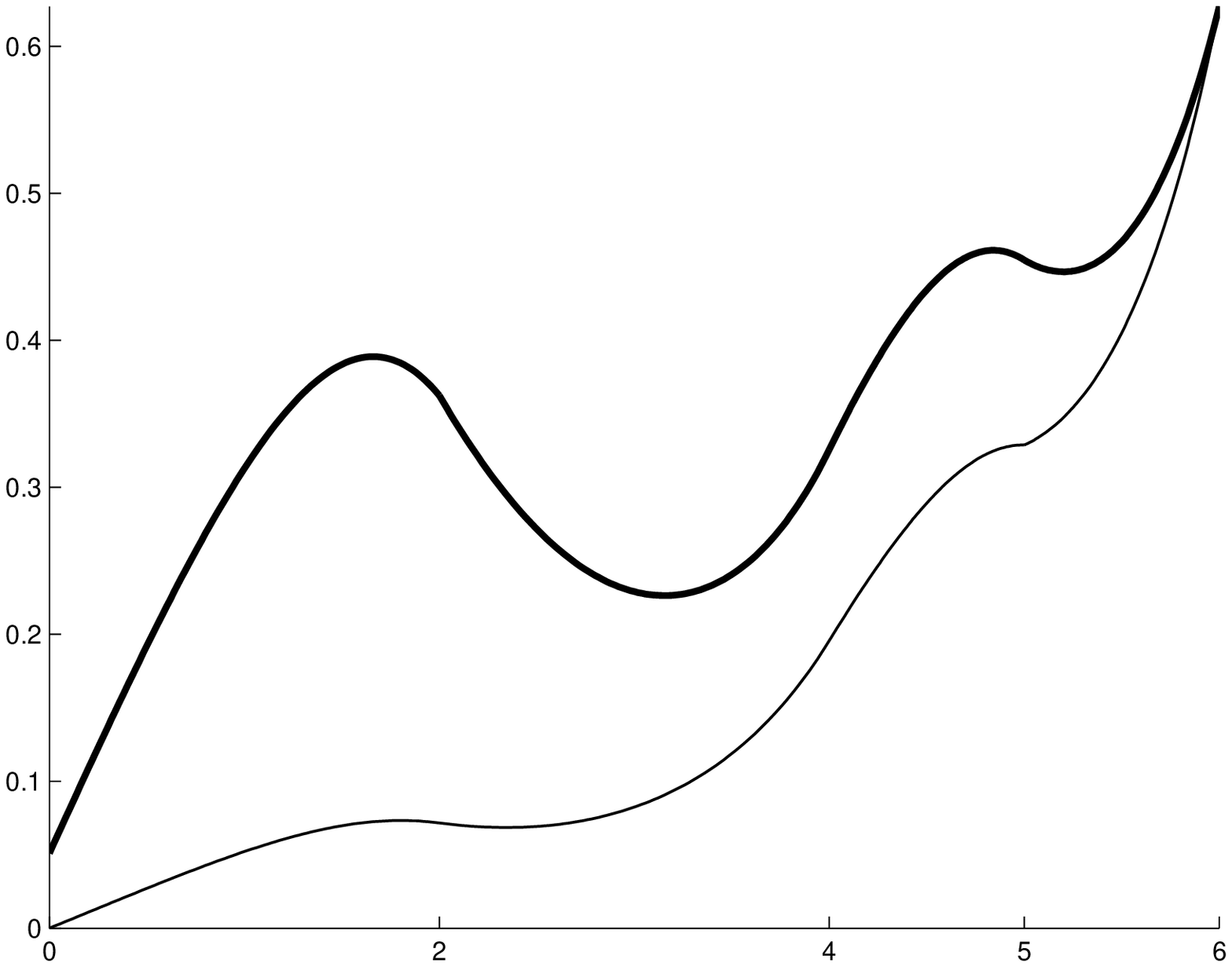}\label{fig:ex1_5}}
\hfill
\subfigure[$\tilde{f}_{1,1}=Df_{2,2}=w_3$]
{\includegraphics[width=0.95\textwidth/3]{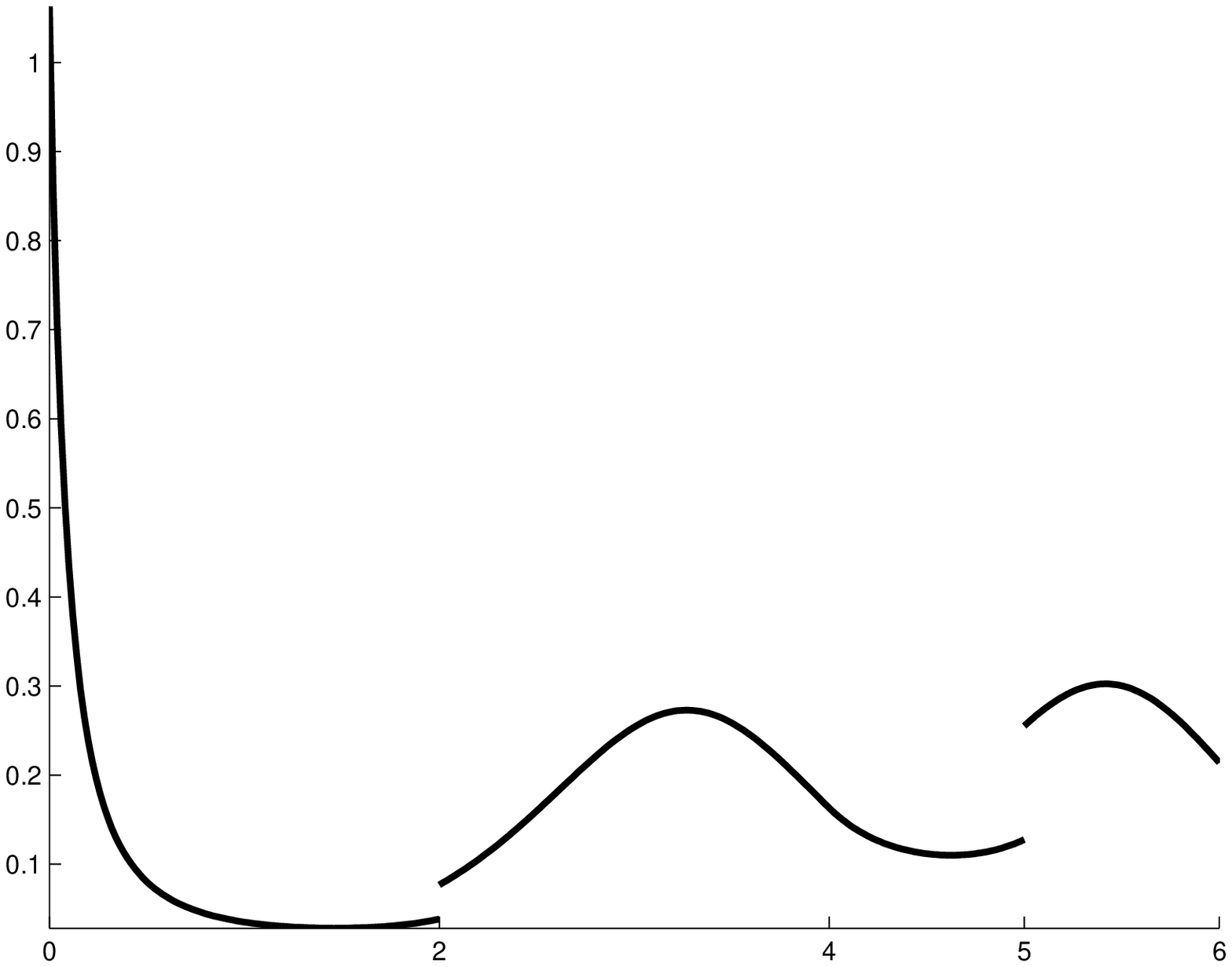}\label{fig:ex1_6}}
\caption{Functions generated from formulas \protect\eqref{eq:wi} and \protect\eqref{eq:fimj} for a spline space $S(\bUU_m,\bDelta,\bR)$, with knots $\{x_0,x_1,x_2,x_3,x_4\}=\{0,2,4,5,6\}$, connection matrices given by equation \protect\eqref{eq:Rex1} and where $\UU_{0,4}=\UU_{2,4}=\Span\{1,x,\cos x, \sin x\}$, $\UU_{1,4}=\UU_{3,4}=\Span\{1,x,\cosh x, \sinh x\}$. In Figures \ref{fig:ex1_4}-\ref{fig:ex1_6} the functions $w_j$ are depicted in bold.}
\label{fig:f1}
\end{figure}

\begin{figure}
\centering
\subfigure[$f_{\ell,4}$, $\ell=1,\dots,4$]
{\includegraphics[width=0.95\textwidth/3]{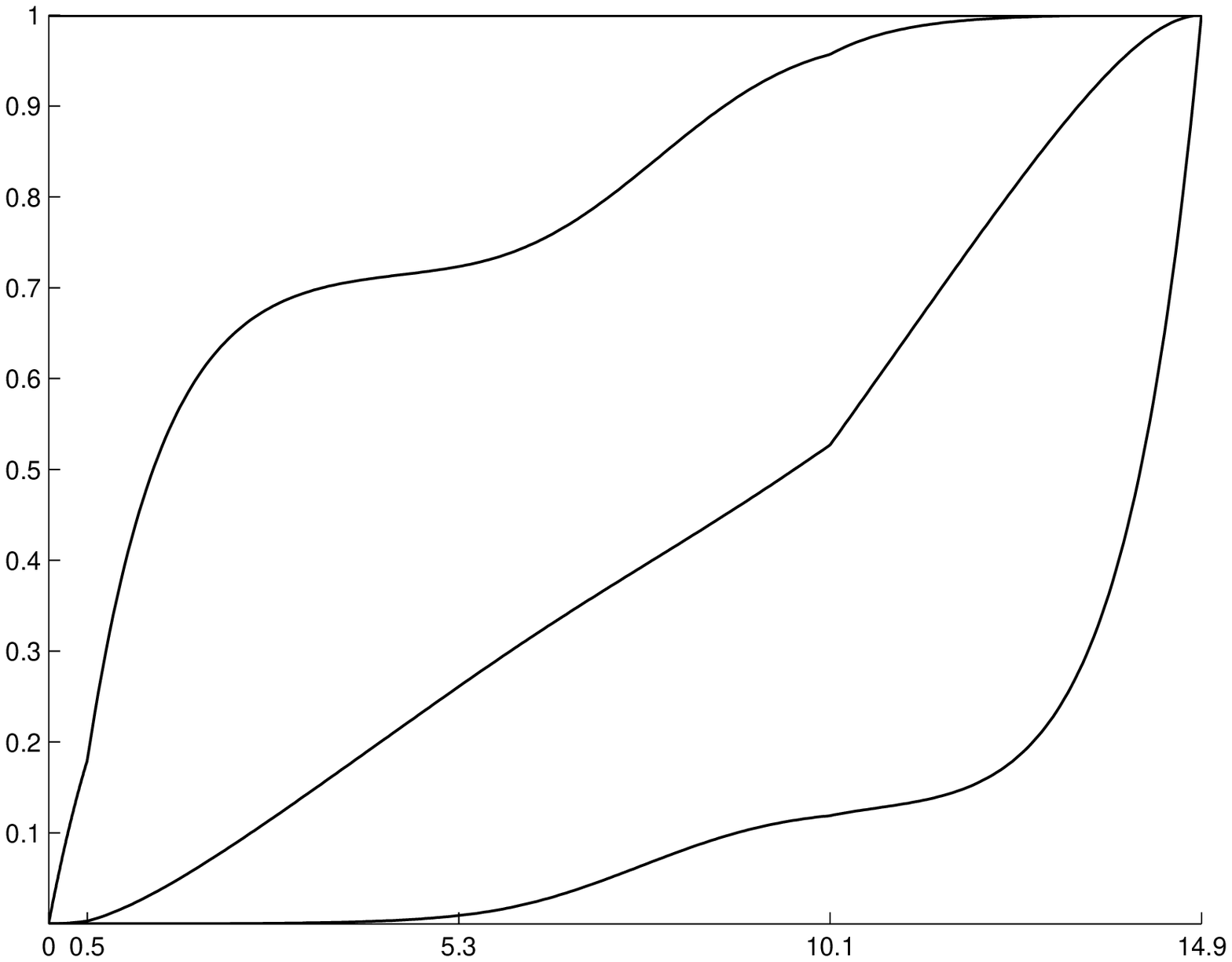}\label{fig:ex1_7}}
\hfill
\subfigure[$\tilde{f}_{\ell,3}=\sum\nolimits_{h=\ell+1}^4 Df_{h,4}$, $\ell=1,2,3$, $w_1=\tilde{f}_{1,3}$]
{\includegraphics[width=0.95\textwidth/3]{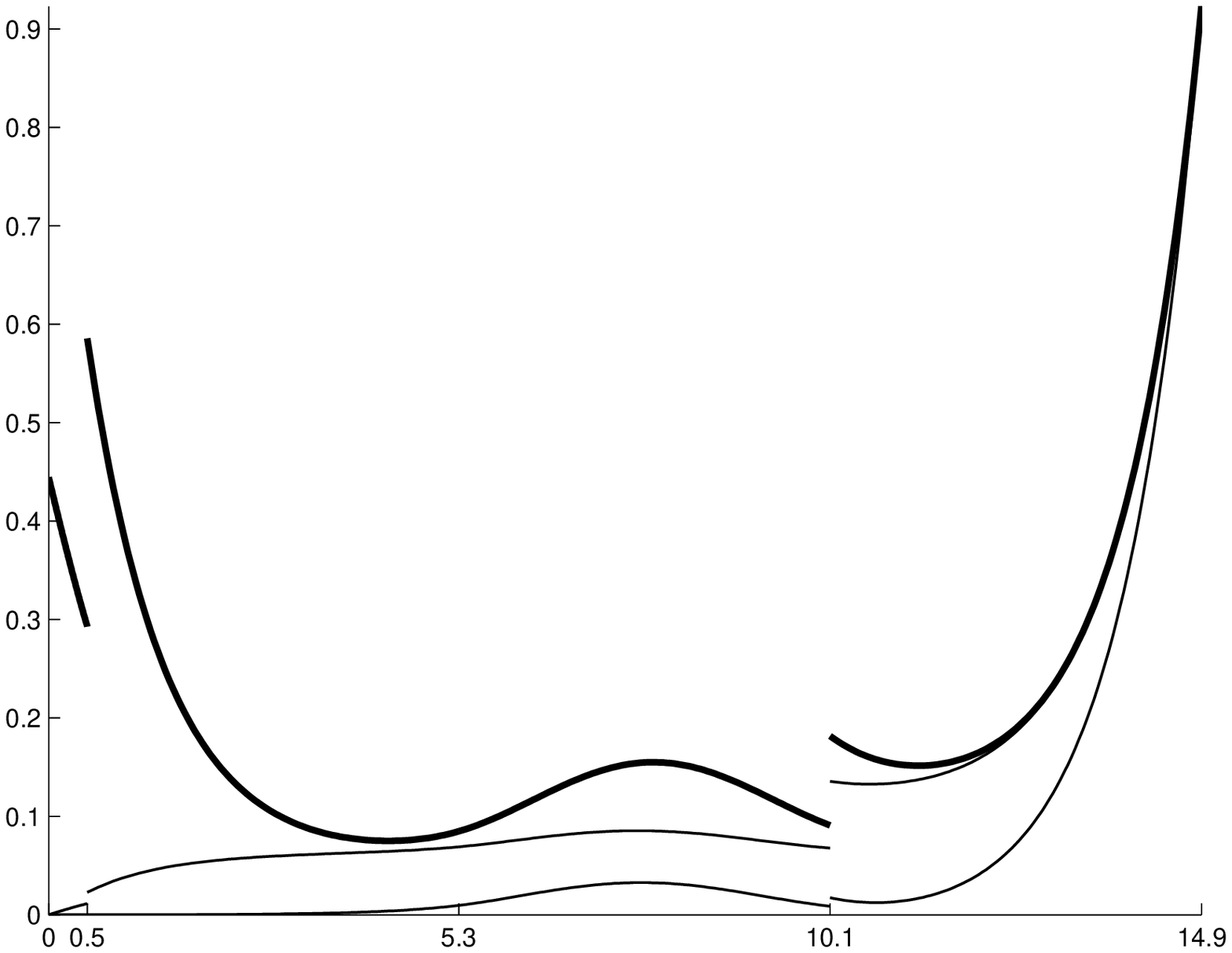}\label{fig:ex1_8}}
\hfill
\subfigure[$f_{\ell,2}$, $\ell=1,\dots,2$]
{\includegraphics[width=0.95\textwidth/3]{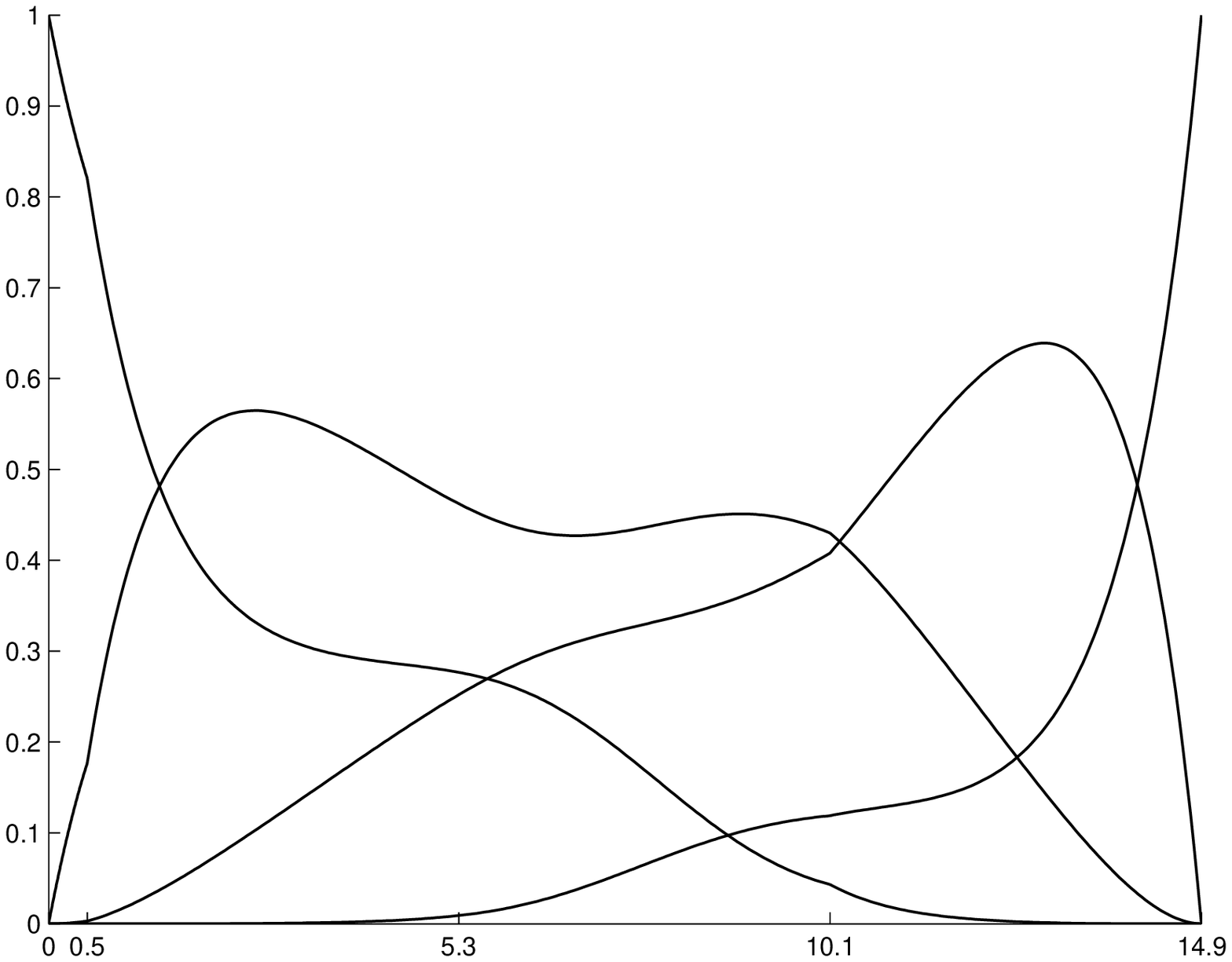}\label{fig:ex1_9}}\\
\subfigure[$f_{\ell,3}$, $\ell=1,\dots,3$]
{\includegraphics[width=0.95\textwidth/3]{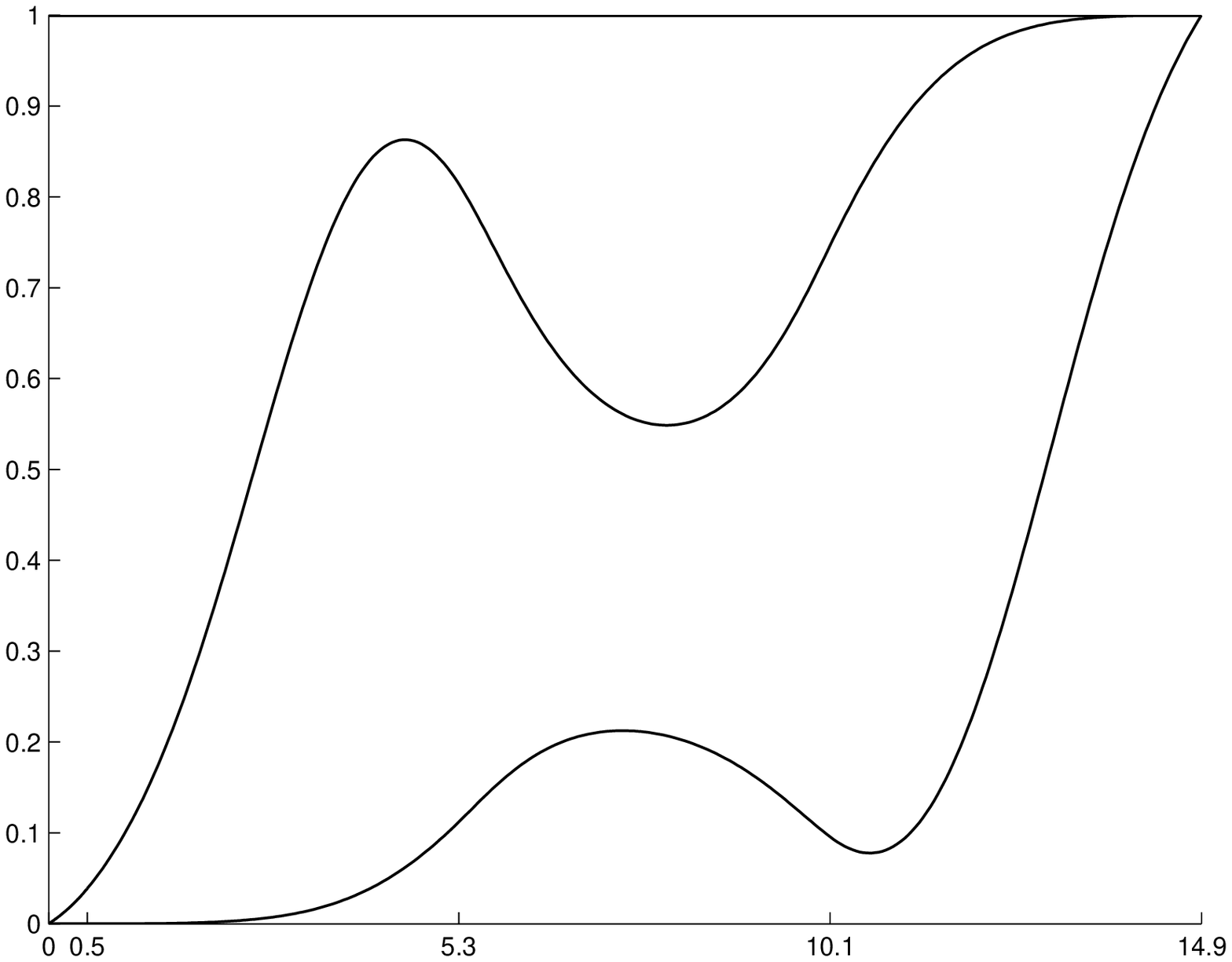}\label{fig:ex1_10}}
\hspace{3ex}
\subfigure[$\tilde{f}_{\ell,2}=\sum\nolimits_{h=\ell+1}^3 Df_{h,3}$, $\ell=1,2$, $w_2=\tilde{f}_{1,2}$]
{\includegraphics[width=0.95\textwidth/3]{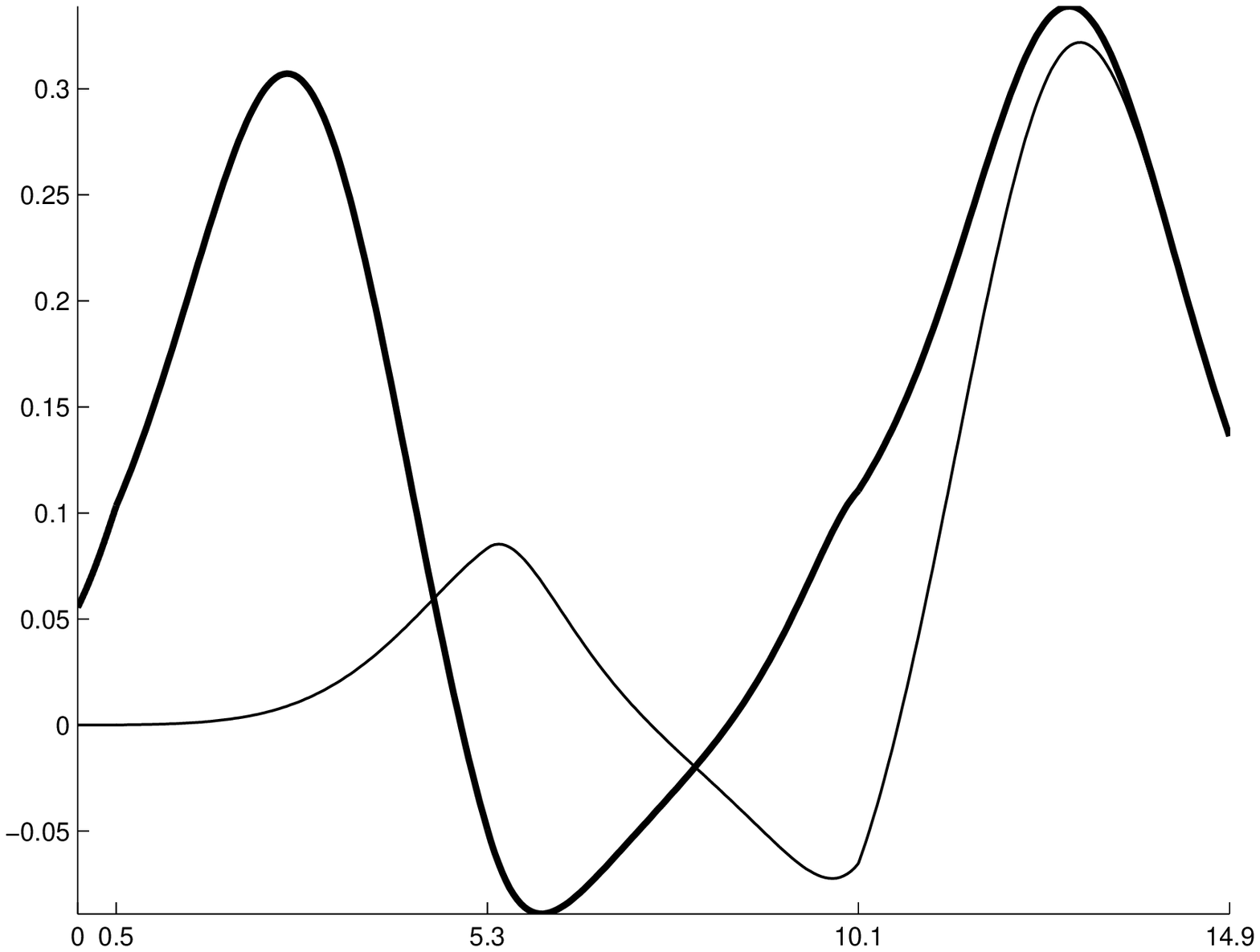}\label{fig:ex1_11}}
\caption{Functions generated from formulas \protect\eqref{eq:wi} and \protect\eqref{eq:fimj} for a spline space $S(\bUU_m,\bDelta,\bR)$, with knots $\{x_0,x_1,x_2,x_3,x_4\}=\{0,0.5,5.3,10.1,14.9\}$, connection matrices given by equation \protect\eqref{eq:Rex1} and where $\UU_{0,m}=\UU_{2,m}=\Span\{1,x,\cos x, \sin x\}$, $\UU_{1,m}=\UU_{3,m}=\Span\{1,x,\cosh x, \sinh x\}$. In Figures \ref{fig:ex1_8} and \ref{fig:ex1_11} the functions $w_j$ are depicted in bold.}
\label{fig:f1_2}
\end{figure}

\begin{figure}
\centering
\subfigure[]
{\includegraphics[width=0.95\textwidth/3]{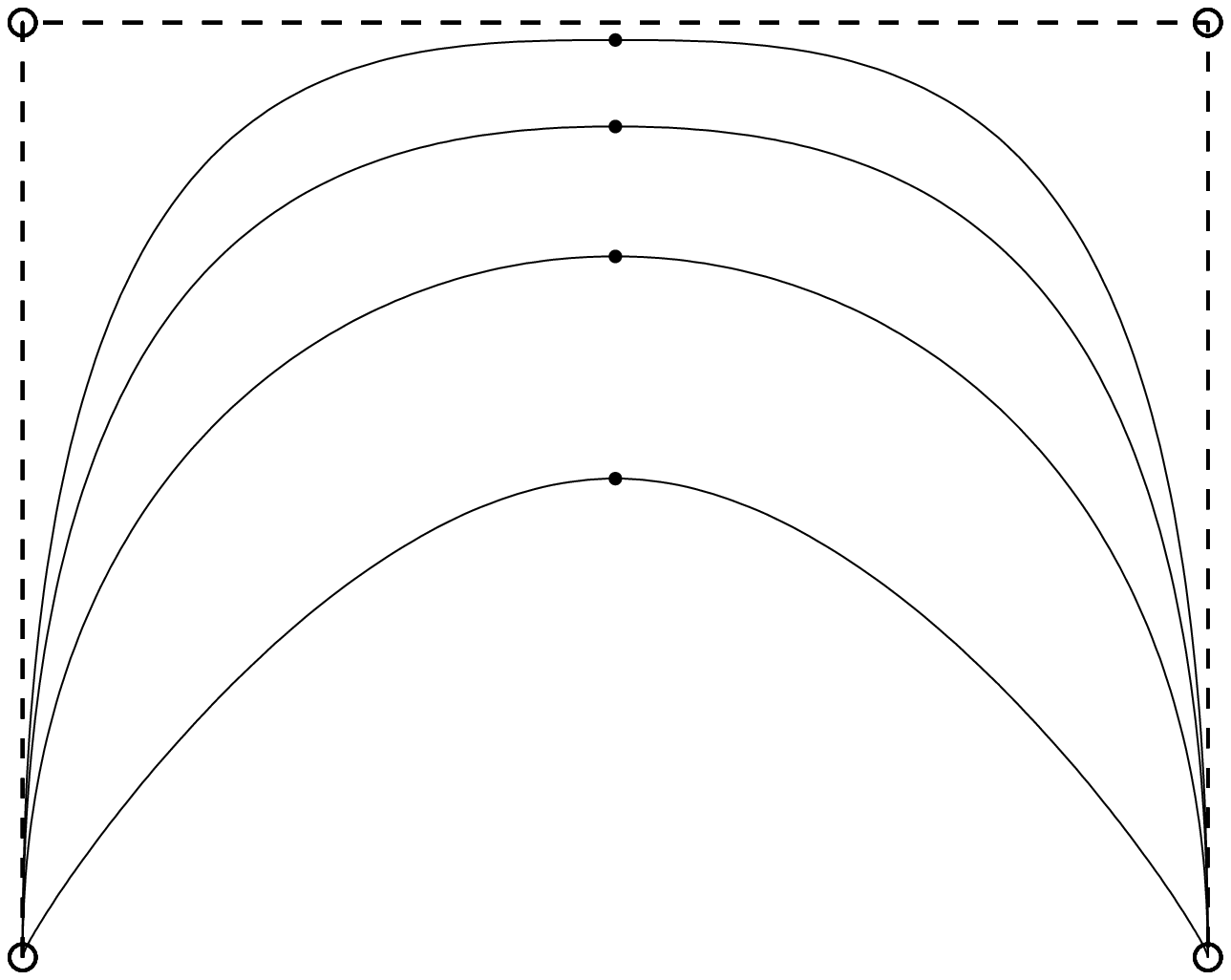}\label{fig:ex2_1}}
\hfill
\subfigure[]
{\includegraphics[width=0.95\textwidth/3]{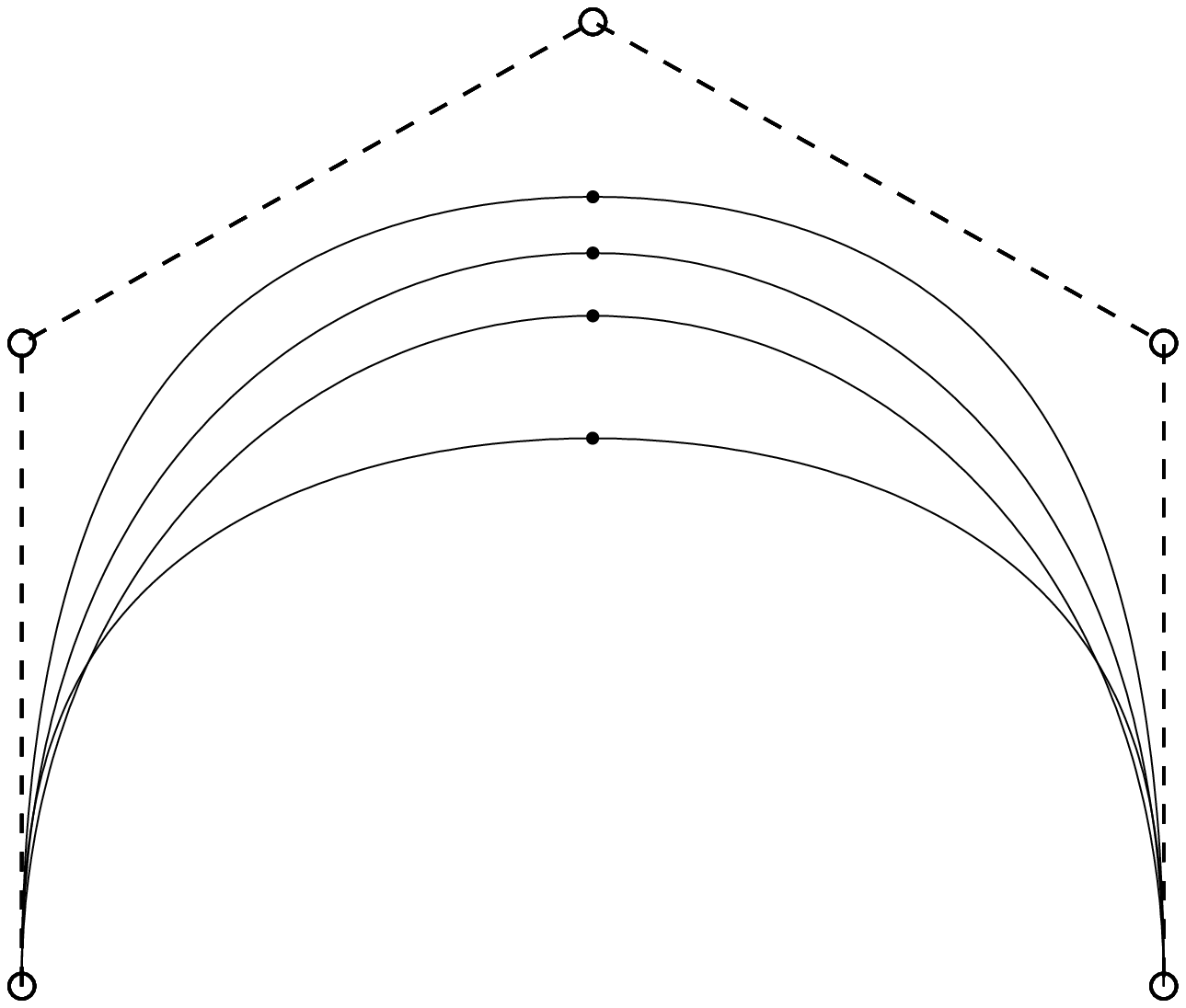}\label{fig:ex2_2}}
\hfill
\subfigure[]
{\includegraphics[width=0.95\textwidth/3]{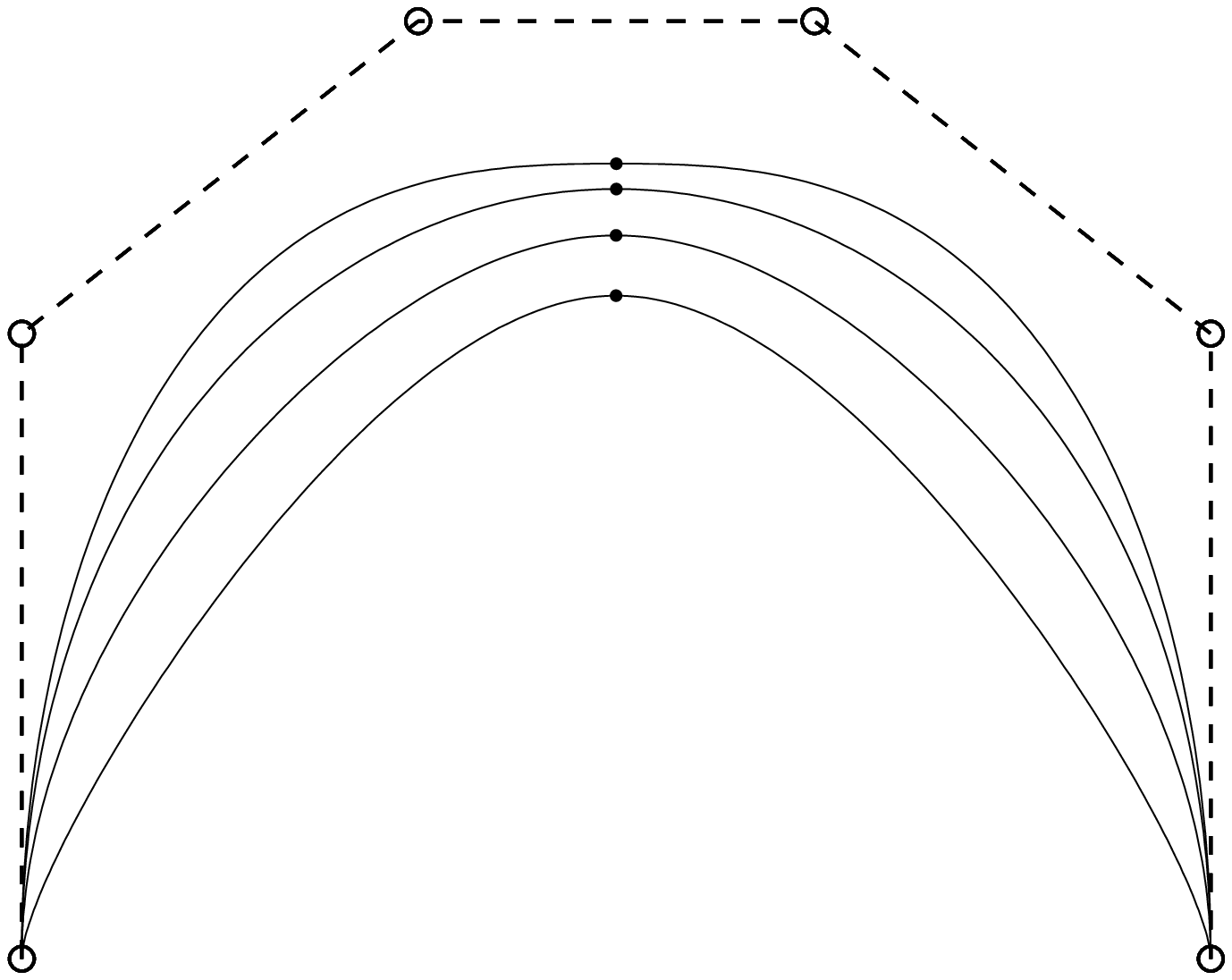}\label{fig:ex2_3}}
\caption{Parametric curves from spline spaces with knots of zero multiplicity. The underlying spline spaces are described in detail in Example \ref{ex:2}.}
\label{fig:f2}
\end{figure}
\end{exmp}

\begin{figure}
\centering
\subfigure[$f_{\ell,4}$, $\ell=1,\dots,4$]
{\includegraphics[width=0.95\textwidth/3]{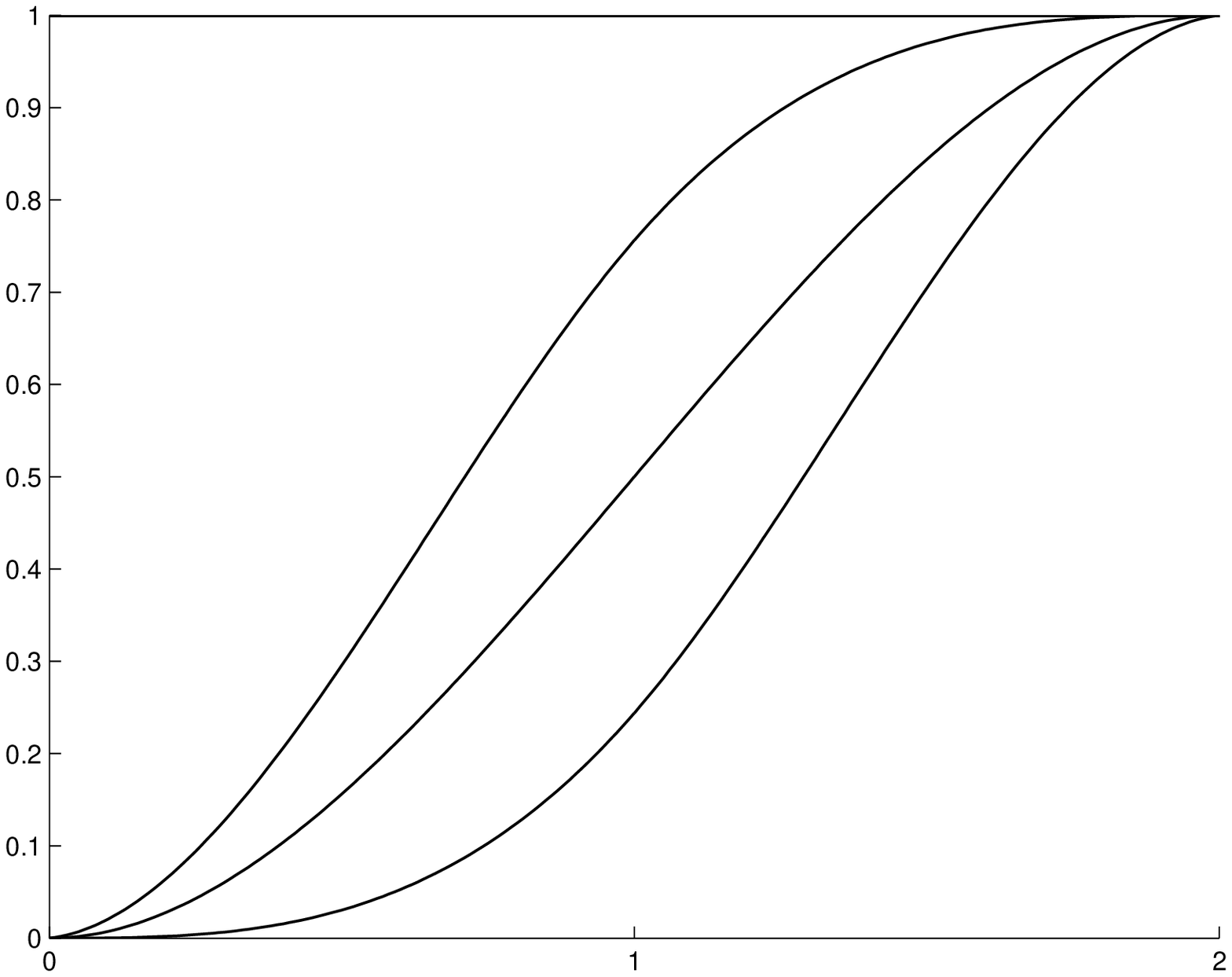}\label{fig:ex2_4}}
\hfill
\subfigure[$f_{\ell,5}$, $\ell=1,\dots,5$]
{\includegraphics[width=0.95\textwidth/3]{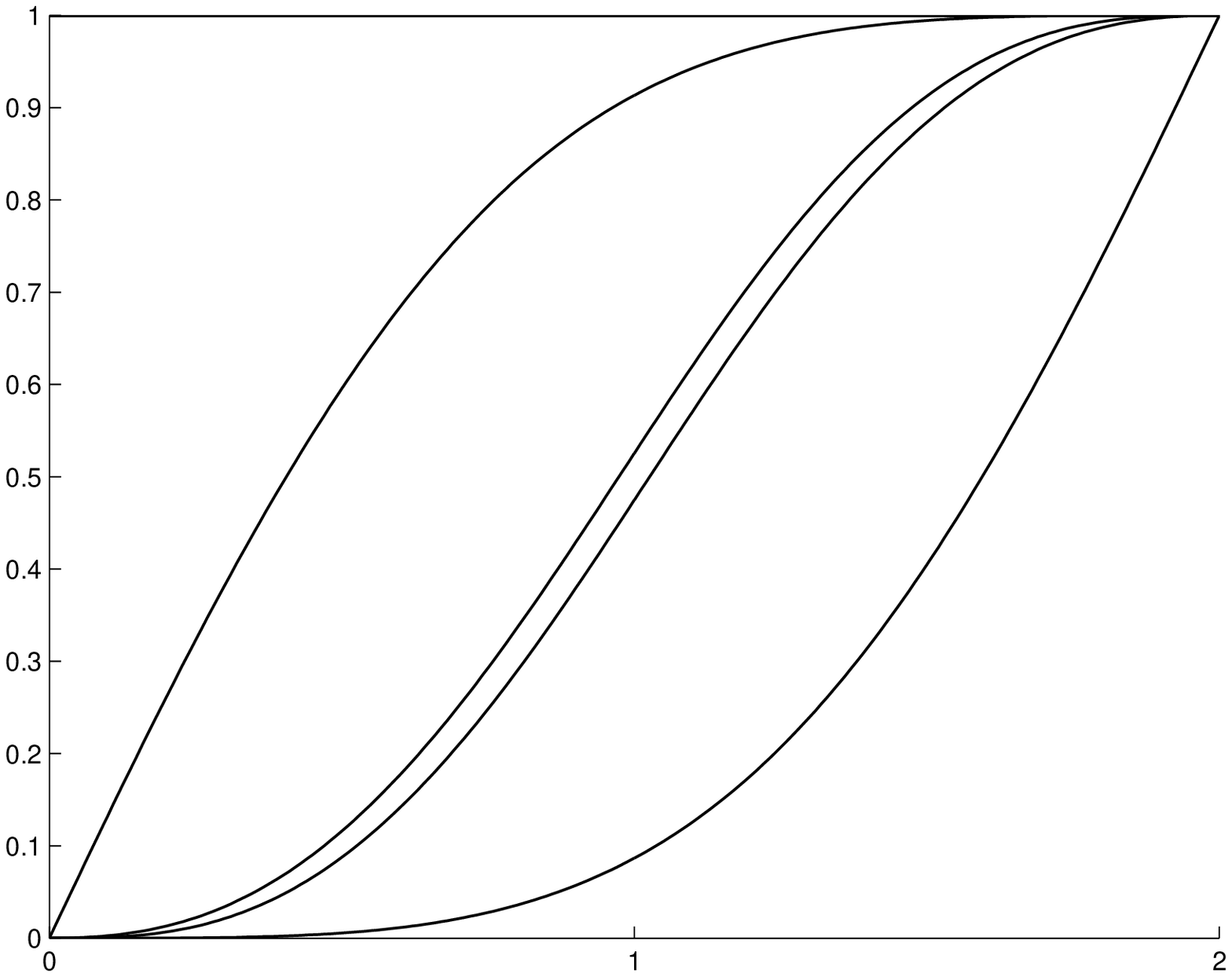}\label{fig:ex2_5}}
\hfill
\subfigure[$f_{\ell,6}$, $\ell=1,\dots,6$]
{\includegraphics[width=0.95\textwidth/3]{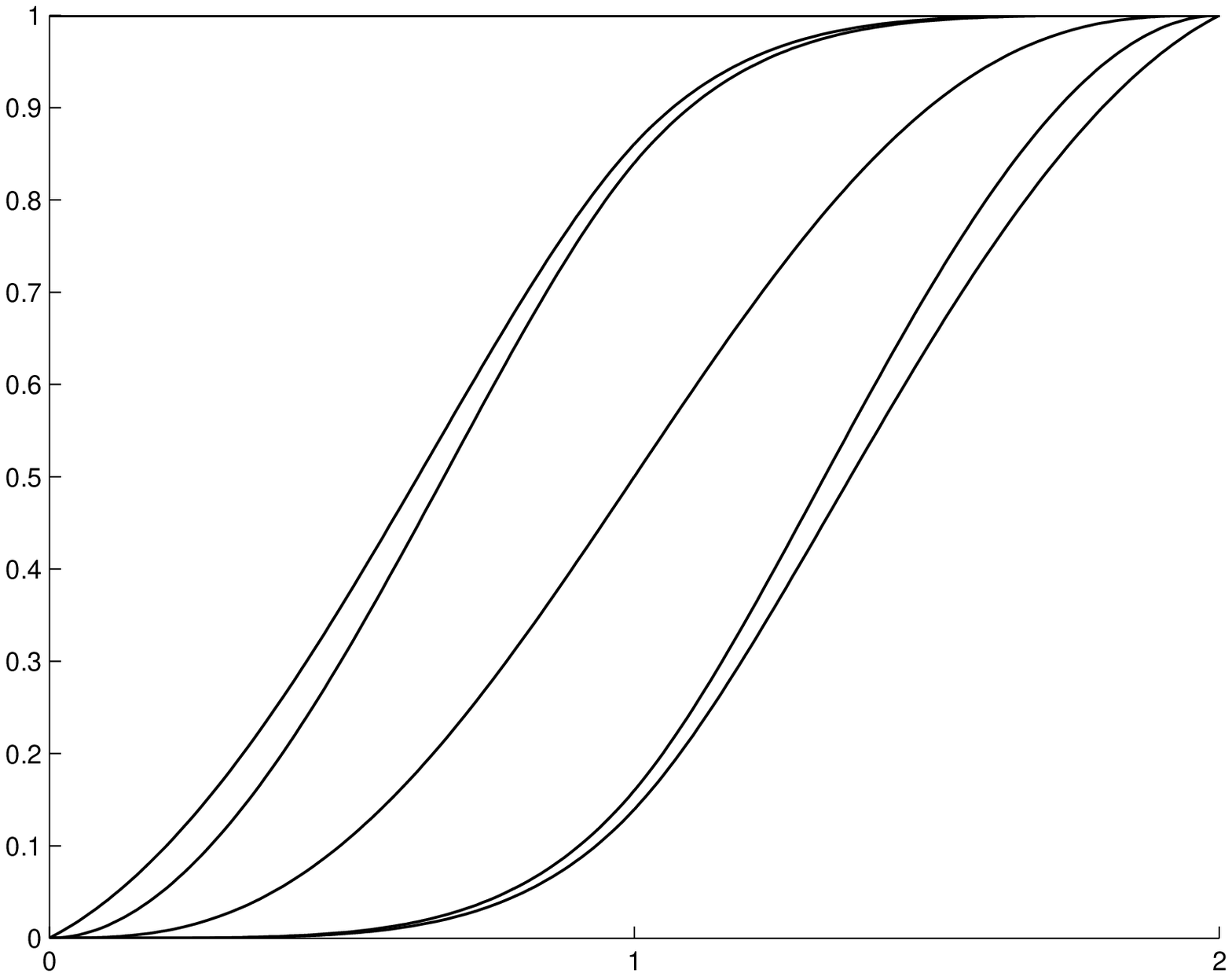}\label{fig:ex2_6}}\\
\subfigure[$B_{\ell,4}$, $\ell=1,\dots,4$]
{\includegraphics[width=0.95\textwidth/3]{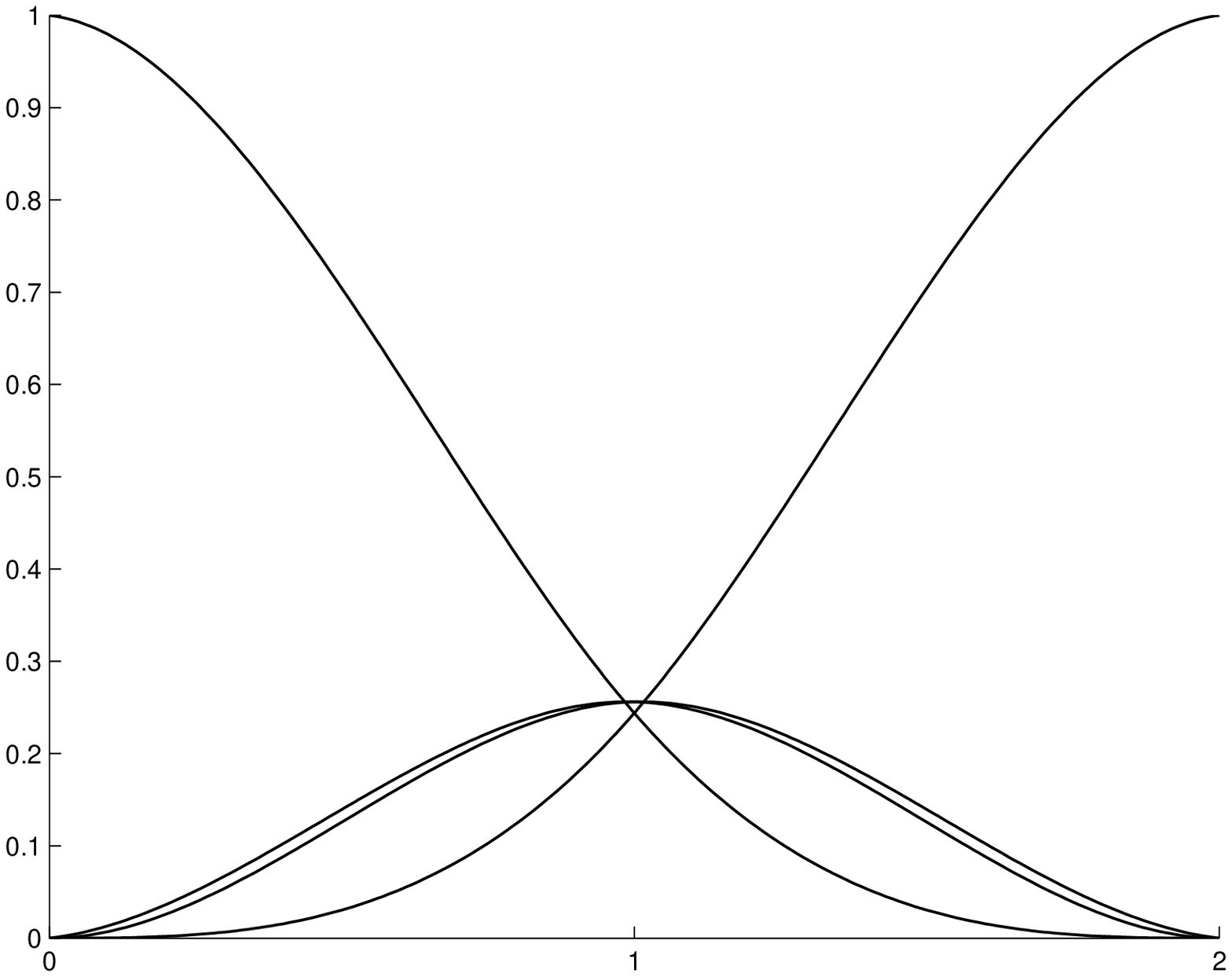}\label{fig:ex2_7}}
\hfill
\subfigure[$B_{\ell,5}$, $\ell=1,\dots,5$]
{\includegraphics[width=0.95\textwidth/3]{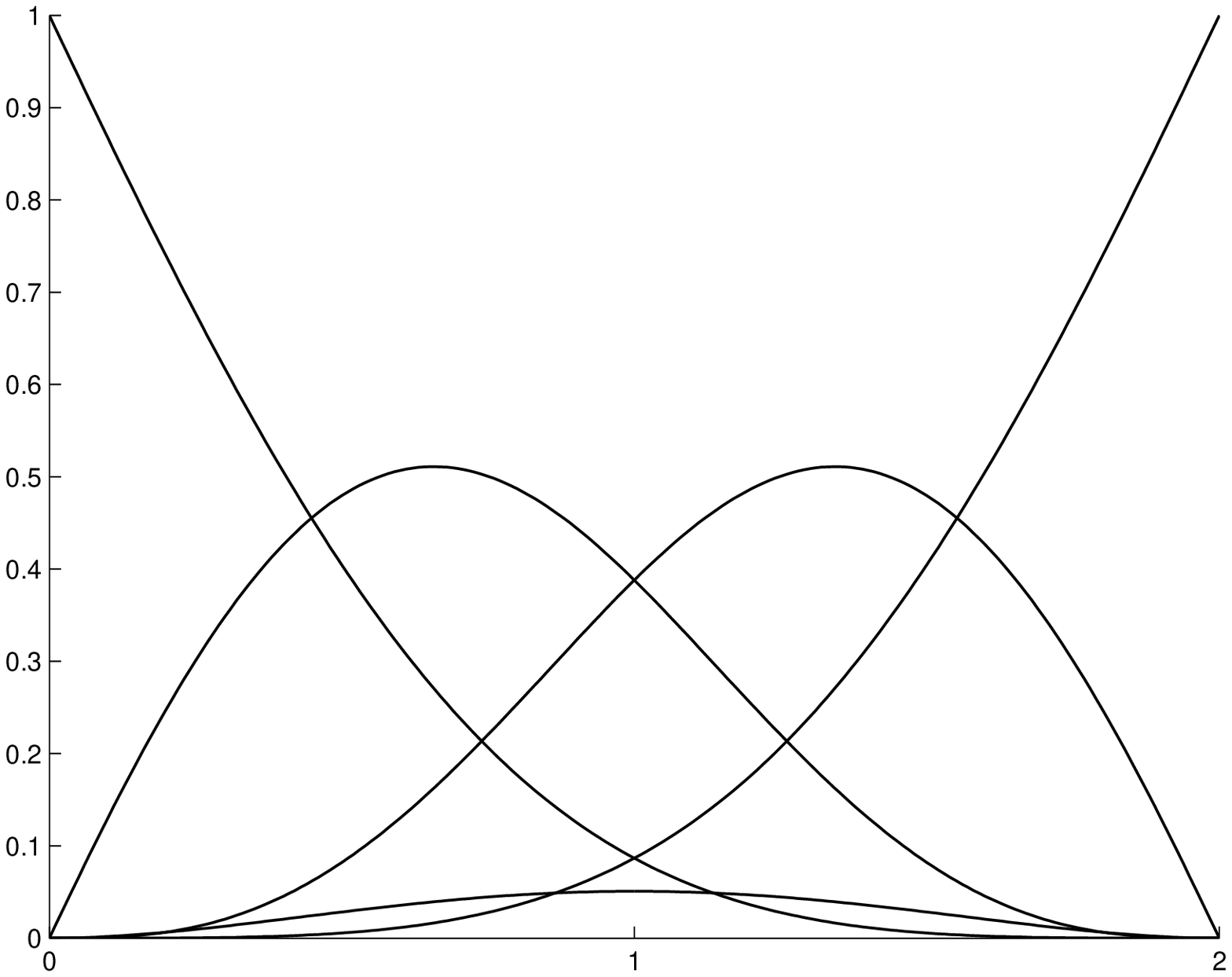}\label{fig:ex2_8}}
\hfill
\subfigure[$B_{\ell,6}$, $\ell=1,\dots,6$]
{\includegraphics[width=0.95\textwidth/3]{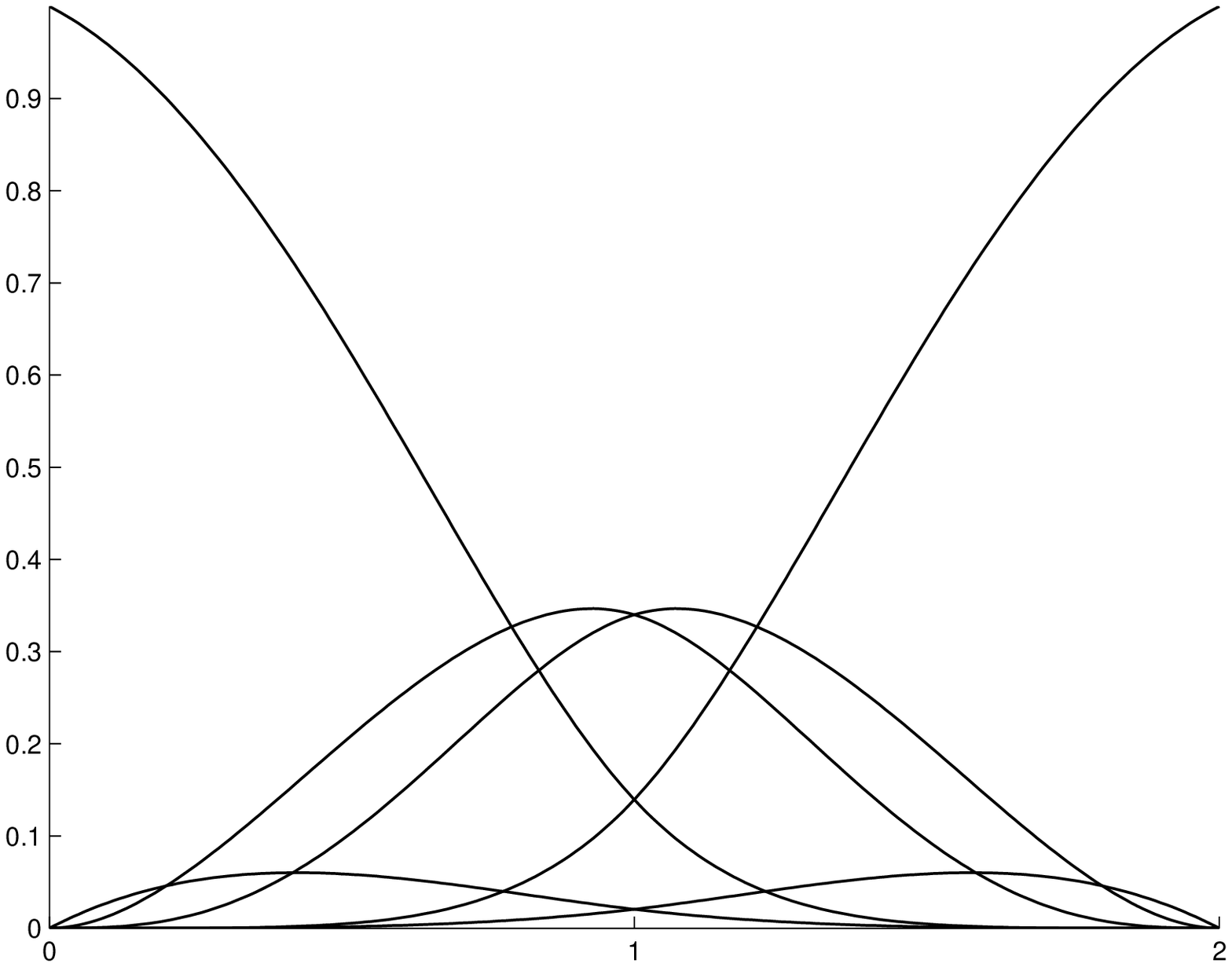}\label{fig:ex2_9}}
\caption{Transition functions and Bernstein basis relative to the spline spaces $S(\bUU_m,\bDelta,\bR)$, $m=4,5,6$, considered in Example \ref{ex:2}. The connection matrices are given in \protect\eqref{eq:R_ex2}. Figure \ref{fig:ex2_4}-\ref{fig:ex2_7}: section spaces $\bUU_{4}$, with connection matrix $R_1^{\text{(a)}}$ where $\beta=-3.9$. Figure \ref{fig:ex2_5}-\ref{fig:ex2_8}: section spaces $\bUU_{5}$, with connection matrix $R_1^{\text{(b)}}$ where $\beta=-3.5$. Figure \ref{fig:ex2_6}-\ref{fig:ex2_9}: section spaces $\bUU_{6}$, with connection matrix $R_1^{\text{(c)}}$ where $\beta=-6.5$.}
\label{fig:f2_2}
\end{figure}

\end{document}